\documentclass[11pt]{article}
\usepackage[margin=1.25in]{geometry}
\usepackage{amsmath, amssymb, amsthm}
\usepackage[sort]{natbib}
\usepackage{algorithm}
\setcitestyle{numbers,square,comma}
\usepackage{tikz,pgfplots}
\usepgfplotslibrary{groupplots}
\usepackage{float}

\usepackage{enumitem}
\usepackage{graphicx}
\usepackage{dcolumn}
\usepackage{bm}
\usepackage{color}
\usepackage{tikz}
\usepackage{pgfplots}
\usepackage{mathtools}
\usepackage{array}
\usepackage[hyphens]{url}
\usepackage{hyperref}
\hypersetup{colorlinks=true,
			urlcolor=black,
			linkcolor=black,
			citecolor=blue,
			bookmarksdepth=paragraph,
			pdftitle={Computation of minimal periods for ordinary differential equations (arXiv version)},
			pdfauthor={Jeremy P. Parker, David Goluskin}}

\usepackage{gnuplottex}

\pgfplotsset{compat=1.18}
\usepackage[noabbrev]{cleveref}

\newcommand\solidrule[1][15pt]{\rule[0.5ex]{#1}{1pt}}
\newcommand\dashedrule{\mbox{%
		\solidrule[3pt]\hspace{3pt}\solidrule[3pt]\hspace{3pt}\solidrule[3pt]}}

\newtheorem{prop}{Proposition}[section]
\newtheorem{rmk}[prop]{Remark}

\newtheorem{conjecture}[prop]{Conjecture}
\newtheorem{lem}[prop]{Lemma}
\theoremstyle{definition}

\crefname{equation}{}{}
\crefname{prop}{proposition}{propositions}
\crefname{lem}{lemma}{lemmas}
\crefname{cor}{corollary}{corollaries}
\crefname{defn}{definition}{definitions}
\crefname{appsec}{appendix}{appendices}
\Crefname{prop}{Proposition}{Propositions}
\Crefname{lem}{Lemma}{Lemmas}
\Crefname{cor}{Corollary}{Corollaries}
\Crefname{defn}{Definition}{Definitions}
\Crefname{appsec}{Appendix}{Appendices}

\newcommand{\ddt}[1]{\frac{\mathrm{d}#1}{\mathrm{d}t}}
\newcommand{\ld}[1]{\mathcal{L}_f #1}
\begin{document}
\title{Computation of minimal periods\\for ordinary differential equations}
\author{Jeremy P. Parker$^1$, David Goluskin$^2$}
\date{$^1$Division of Mathematics, University of Dundee, Dundee, DD1 4HN, United Kingdom\\
$^2$Department of Mathematics and Statistics, University of Victoria, Victoria, BC, V8P 5C2, Canada}

\maketitle

\begin{abstract}
A framework is presented for lower-bounding periods among periodic solutions to an autonomous dynamical system governed by ordinary differential equations. For a chosen dynamical system, lower bounds can be proved by constructing auxiliary functions that, similarly to Lyapunov functions, satisfy a certain inequality pointwise on state space. Different formulations can give bounds applying either to all periodic solutions or to only periodic solutions with chosen symmetry. In the case of differential equations that are polynomial in the state variables, we present computational methods that use semidefinite programming to construct auxiliary functions. Furthermore, we give an algorithm to rigorously validate the numerically computed bounds via rational arithmetic. To illustrate these methods, computations are carried out for two chaotic systems that each have an infinite number of periodic solutions: the Lorenz system, which is dissipative, and the H\'enon--Heiles system, which is Hamiltonian. All computed bounds are validated with rational arithmetic. Separate bounds are computed that apply to all periodic solutions, and to only periodic solutions with certain symmetries. In all cases, our best validated bounds agree with periods of known periodic solutions to at least 5 digits, which strongly suggests exact sharpness of our framework for these examples. The question of how broadly our framework is sharp is discussed, but it remains open.
\end{abstract}

\maketitle

\section{Introduction}

Periodic orbits (POs) that satisfy autonomous systems of ordinary differential equations (ODEs) amount to closed curves in state space that integrate the vector field defined by the ODE's right-hand side. A strictly positive lower bound on the periods of all POs can be deduced if the vector field has certain regularity: upper bounds on the speed and curvature along every ODE solution naturally restrict how quickly an orbit can close on itself. This intuition was made quantitative by \citet{yorke1969periods}, who proved a lower bound on the period of a PO $x(t)$ evolving in $\mathbb{R}^n$ according  $\ddt{}x(t)=f(x(t))$. If the function $f$ is globally Lipschitz continuous with constant $L$,
% meaning that $|f(x)-f(y)|\leq L|x-y|$ for all $x,y\in\mathbb{R}^n$, 
then any PO must have a period of at least $2\pi/L$. Yorke's argument uses Fenchel's theorem, which states that any smooth closed curve has total curvature of at least $2\pi$. The argument also uses a pointwise worst-case bound on curvature in terms of $L$, but in no other way does it use the fact that trajectories integrate the $f(x)$ vector field.

Yorke's argument \citep{yorke1969periods} has been generalized in various ways, including to dynamics in Hilbert spaces and Banach spaces \citep{busenberg1986better,robinson2006minimal,zevin2009minimal,cheng2009generalized,zevin2012minimal,nieuwenhuis2014minimal, herzog2024minimal}, and to prove a positive lower bound on the period of solutions to the 3D Navier--Stokes equations   \citep{kukavica1994absence}. On Hilbert spaces, the lower bound $2\pi/L$ is sharp in the sense that there exists a vector field having Lipschitz constant $L$ and a PO with period $2\pi/L$. For any particular dynamical system, however, the bound is typically not sharp, meaning that the infimum of periods among all POs is strictly larger than $2\pi/L$.

The present work develops methods for proving sharper lower bounds on periods of POs for any particular ODE system. Rather than using a pointwise worst-case estimate on curvature, such as Yorke's estimate by the Lipschitz constant, we make more use of the fact that orbits must integrate the vector field $f(x)$. The approach is inspired by observations that (1) total curvature of a PO can be expressed as the time average of a certain quantity along that orbit, and (2) methods have recently been developed whereby sharp bounds on time averages in ODE systems can be computed using polynomial optimization, in particular with sum-of-squares constraints \citep{chernyshenko2014polynomial,fantuzzi2016bounds}. Our methods can be used either to prove bounds that apply to all POs, or to only POs that possess a desired symmetry. Previous results by Yorke and others based on the Lipschitz constant are recovered as suboptimal special cases.

Our approach relies on constructing an auxiliary function, defined on the state space, subject to a certain pointwise inequality that implies the desired bound on orbit period. This is similar to the use of Lyapunov functions to verify stability, although they are subject to different pointwise inequalities that imply different conclusions. The general version of our framework is broadly applicable to ODEs, at least in principle, but in practice it can be very hard to carry out. However, in the case where the vector field $f(x)$ is polynomial, the search for an auxiliary function subject to the desired inequality is more tractable. In particular, the inequality can be enforced by a polynomial sum-of-squares (SOS) constraint, and the resulting SOS optimization problem can be recast as a semidefinite program (SDP). An SDP is a standard type of convex optimization problem that can be solved numerically using floating point arithmetic. Furthermore, if a computer-assisted proof is desired, numerical SOS or SDP solutions can be subsequently validated using interval arithmetic \citep{jansson2008rigorous,goluskin2018bounding,parker2024lorenz} or rational arithmetic \citep{peyrl2008computing,kaltofen2012exact,magron2021exact,dostert2021exact,davis2024rational}. Here we validate SDP solutions using rational arithmetic, by a strategy similar to that of \citet{peyrl2008computing}.

One motivation for determining the shortest period among POs is to confirm, after various POs have been found numerically, that the shortest PO is among them. Finding the shortest PO, or several of the shortest POs with different symmetries, can be useful for several reasons. First, a short PO may constitute a good control objective in some applications, such as in the design of orbital missions for spacecraft. Second, the unstable periodic orbits embedded in a chaotic attractor offer ways to study the chaotic dynamics, including by expanding chaotic averages as weighted sums over the POs \citep{cvitanovic1991periodic}. The shortest POs are the most heavily weighted in such expansions. They are also important when searching for trajectories that maximize the time averages of various quantities \citep{yang2000optimal,goluskin2018bounding}.

In \cref{sec:thm}, we derive the inequality on state space that implies a lower bound on the periods of POs. We also show how this inequality can be enforced by an SDP constraint in the case of polynomial ODEs, and how bounds can be made specific to POs with desired symmetry. In \cref{sec:examples}, our computational approach based on SDPs is used to prove computer-assisted lower bounds on orbit periods in the H\'enon--Heiles system and the Lorenz system. The first system displays Hamiltonian chaos, and the second displays dissipative chaos. For each system, our computations (with code provided) validate lower bounds on period are sharp to five digits, as confirmed by known POs whose periods match the bounds to this precision. \Cref{sec:conc} gives concluding remarks.

\section{Convex problems giving lower bounds on periods}
\label{sec:thm}

In what follows we consider ODE dynamical systems
\begin{equation}
        \ddt{}x(t) = f(x(t)) \quad \text{on}\quad \Omega\subset\mathbb{R}^n,
    \label{eq:system}
\end{equation}
where $f\in C^1(\Omega,\mathbb{R}^n)$, meaning that $f$ is a continuously differentiable mapping from $\Omega$ to $\mathbb{R}^n$. The differentiability of $f$ ensures that trajectories $x(t)$ of \cref{eq:system} are unique and $C^1$. We assume that $x(t)$ remains in $\Omega$ for all $t\ge 0$ -- i.e., that $\Omega$ is forward invariant under the flow generated by \cref{eq:system}. In general we do not assume that $\Omega$ is bounded, nor that each $x(t)$ is bounded uniformly for positive times.

This section presents our main theoretical results and methods. \Cref{sec: general bound} gives the general framework for lower-bounding periods of POs of~\cref{eq:system}. \Cref{sec: SDP general} derives SDPs to compute these bounds in the case where $f(x)$ is polynomial. For ODEs with symmetries, \cref{sec: sym} explains how to exploit symmetries computationally, and how to find bounds that apply only to POs that are invariant under a symmetry.

\subsection{Variational problems over differentiable auxiliary functions}
\label{sec: general bound}

The conditions we derive for lower-bounding periods of POs rely on three ingredients. The first is that one can construct a vector function $\Phi(x)$ of any dimension, defined on $\Omega$, such that each component of $\Phi$ time-integrates to zero over each PO of interest. If $\Phi=f \cdot D W$ for any $W\in C^1(\Omega,\mathbb{R}^m)$, then it has zero mean over any PO; this follows from integrating the identity $\Phi(x(t))=\tfrac{\rm d}{{\rm d}t}W(x(t))$ over one period. Additional mean-zero $\Phi$ can be constructed using symmetry principles if the POs of interest have a symmetry, as described later. The second ingredient, which also underlies previous lower bounds on orbit period, is that mean-zero functions satisfy Wirtinger's inequality. \Cref{lem:wirtinger} below states the simple version of Wirtinger's inequality on intervals that we require, wherein $|\cdot|$ denotes the usual Euclidean norm. Note that this lemma concerns general functions $\Phi(t)$ that are not necessarily induced by ODEs. We do not include a proof since it follows directly from the special case with $T=2\pi$ and $m=1$ that is proved in standard references \citep[section 7.7]{hardy1952inequalities}.

\begin{lem}[Wirtinger's inequality]
\label{lem:wirtinger}
For any $T>0$ and $\Phi\in  C^1([0,T],\mathbb{R}^m)$ with $\int_0^T \Phi(t)\,\mathrm{d}t = 0$ and $\Phi(T)=\Phi(0)$,
\begin{equation}
    \int_0^T |\Phi(t)|^2 \mathrm{d}t \leq \left(\frac{T}{2\pi}\right)^2 \int_0^T \left|\ddt{\Phi}\right|^2 \mathrm{d}t.
\end{equation}
\end{lem}

The third ingredient is a method for bounding time averages in ODE systems by constructing auxiliary functions $V\in C^1(\Omega)$ that take values in $\mathbb{R}$ and satisfy a certain inequality pointwise on phase space \citep{fantuzzi2016bounds}. With these three ingredients we obtain our general framework for bounding orbit periods, which is stated by \cref{thm:main} and explained in the proof. We make repeated use of the notation $\ld$ for the Lie derivative with respect to the flow generated by $f$, whose application to $\Phi$ gives
\begin{equation}
 \ld \Phi (x)=f(x)\cdot D\Phi(x) = \sum_{i=1}^xf_i(x)\frac{\partial\Phi}{\partial x_i},
\end{equation} 
where $D\Phi$ is the gradient of $\Phi\in C^1(\Omega)$ with respect to $x$, and where the dot denotes a sum over the components of $f_i$, as made explicit by the second equality. Crucially, $\ld \Phi$ can be evaluated without knowing any trajectories, yet it satisfies $\ddt{} \Phi(x(t)) = \ld\Phi(x(t))$ for any trajectory $x(t)$ of~\cref{eq:system}.

\begin{prop}
\label{thm:main}
Consider $\ddt{}x(t) = f(x(t))$ with $f\in C^1(\Omega,\mathbb{R}^n)$ and $\Omega\subset\mathbb{R}^n$ forward invariant. Suppose there exist $B>0$, $V\in C^1(\Omega, \mathbb{R})$ and $\Phi\in C^1(\Omega, \mathbb{R}^m)$ such that
\begin{equation}
    B\left| \Phi(x) \right|^2 - \left| \ld \Phi(x)\right|^2 + \ld V(x) \geq 0 \quad \forall~x\in\Omega.
\label{eq:inequality}
\end{equation}
Suppose $x(t)$ is $T$-periodic for some $T>0$. If $\Phi(x)$ satisfies the mean-zero condition $\int_0^T\Phi(x(t))\,\mathrm{d}t=0$ but does not vanish everywhere on $x(t)$, then 
\begin{equation}
T\geq 2\pi/\sqrt{B}. 
\label{eq:general bound}
\end{equation}
\end{prop}

\begin{proof}
Suppose that $B$, $V$ and $\Phi$ satisfy the assumptions, and that $x(t)$ is a $T$-periodic orbit on which $\Phi$ satisfies the assumptions. The inequality~\cref{eq:inequality} means that
    \begin{equation}
    B\left|\Phi(x(t)) \right|^2 - \left| \ddt{}\Phi(x(t))\right|^2 + \ddt{}V(x(t)) \geq 0
    \label{eq:inequality2}
    \end{equation}
holds along ODE trajectories at all $t$. Integrating over one period gives
    \begin{equation}
    B\int_0^T \left| \Phi(x(t)) \right|^2 \mathrm{d}t - \int_0^T \left| \ddt{}\Phi(x(t)) \right|^2 \mathrm{d}t\geq 0,
    \end{equation}
where the last term in~\cref{eq:inequality2} has vanished since $V(x(t))$ is $T$-periodic. Rearranging gives
    \begin{equation}
        \frac{\int_0^T \left| \ddt{}\Phi(x(t)) \right|^2 \mathrm{d}t}{\int_0^T \left| \Phi(x(t)) \right|^2 \mathrm{d}t} \leq B,
        \label{eq:inequality3}
    \end{equation}
where the denominator is strictly positive under the assumption that $\Phi(x(t))$ does not vanish identically. (Note that $t\mapsto\Phi$ is continuous.) Because $\Phi(x(t))$ is assumed to be mean-zero over period $T$, \cref{lem:wirtinger} gives
    \begin{equation}
        \left(\frac{2\pi}{T}\right)^2 \leq \frac{\int_0^T \left| \ddt{}\Phi(x(t)) \right|^2 \mathrm{d}t}{\int_0^T \left| \Phi(x(t)) \right|^2 \mathrm{d}t}.
        \label{eq:inequality4}
    \end{equation}
Combining \cref{eq:inequality3,eq:inequality4} and rearranging gives the claim~\cref{eq:general bound}.
\end{proof}

The lower bound~\cref{eq:general bound} from \cref{thm:main} is largest when $B$ is smallest. One naturally hopes to optimize the bound over $\Phi$ and $V$ for which the proposition's assumptions holds. A simple suboptimal choice is to let $V=0$ identically, and to let $\Phi=f$, which indeed has zero mean on all POs without vanishing identically. In this case, if $f$ is globally Lipschitz continuous with constant $L$, then the operator norm of the Jacobian is bounded by $L$, and so $|\ld f| \leq \left\|Df\right\| |f| \leq L |f|$  implies that the assumptions of \cref{thm:main} are satisfied with $B=L^2$. The resulting bound is $T\geq 2\pi/L$ for any periodic orbit, which is exactly the seminal result of \citet{yorke1969periods}. The additional freedom to optimize over nonzero $V\neq0$ and other $\Phi$ will give sharper bounds in general, including in cases where $f$ is not globally Lipschitz continuous. Such optimization is the object of the next subsection.

\subsection{Semidefinite programs for polynomial auxiliary functions}
\label{sec: SDP general}

We wish to maximize the lower bound~\cref{eq:general bound} over some class of $\Phi$ and $V$ that satisfy the assumptions of \cref{thm:main}. Imposing the assumptions exactly in such an optimization is generally not possible because it would require knowing the ODE trajectories. Instead we impose stronger but more tractable conditions that imply the assumptions of \cref{thm:main}. 

The pointwise inequality~\cref{eq:inequality} is not convex in $\Phi$, which is an impediment to efficient optimization, but we can replace $\Phi$ with a matrix variable $\mathsf Q$ that will appear affinely in the constraint. To do this, we choose an ordered basis $a(x)=(a_1(x),\ldots,a_m(x))$ whose components $a_i$ each satisfy the conditions for $\Phi$ -- they each have zero mean on a given class of POs but do not all vanish identically. We then seek each component of $\Phi$ from the span of this basis, meaning $a\in C^1(\Omega,\mathbb{R}^m)$ and $\Phi(x)=\mathsf{U}a(x)$ for some square matrix $\mathsf{U}$. This implies that $|\Phi(x)|^2=a(x)^\mathsf{T}\mathsf{Q}a(x)$ for an $m\times m$ symmetric matrix $\mathsf{Q}\succeq0$, where $^\mathsf T$ denotes the transpose and $\succeq$ denotes positive semidefiniteness. We further require that $\mathsf Q$ is strictly positive definite, denoted $\mathsf{Q}\succ0$, so that the $\Phi(x)$ vector vanishes only where the basis $a(x)$ does. After choosing $a$, optimizing $\Phi$ amounts to optimizing $\mathsf Q\succ0$. The resulting optimization problem is formulated as~\cref{eq:Ba def} in the first part of \cref{thm:general} below. Its constraints are easier to impose than the assumptions of \cref{thm:main}, but the problem is typically still intractable. Even if $V$ were fixed, rather than optimized over the infinite-dimensional space $C^1$, verifying the pointwise constraint in~\cref{eq:Ba def} would be intractable in general. This can be overcome by further strengthening the constraints, at least when all expressions are polynomial.

For the remainder of this subsection, we suppose that the ODE right-hand side $f(x)$ is polynomial in the components of $x$, and we choose the basis vector $a(x)$ and the auxiliary function $V(x)$ to be polynomial also. In this case, the constraint in the optimization problem~\cref{eq:Ba def} amounts to nonnegativity of a polynomial expression on the set $\Omega$. Such a constraint can be strengthened into a more tractable SOS condition in a standard way that we now describe. Our description in terms of SOS polynomials has a dual perspective in terms of moments of measures~\cite{lasserre2009moments}, and in the present case the relevant measures are dynamically invariant as in~\cite{korda2021convex}.

For a polynomial $p:\mathbb{R}^n\to\mathbb{R}$ to satisfy $p(x)\ge0$ for all $x\in\mathbb{R}^n$, it suffices that $p\in \Sigma$, where $\Sigma$ denotes the set of polynomials that can be represented as a sum of squares of other polynomials. 
%The set of nonnegative polynomials strictly contains $\Sigma$, and the relationship between these sets has been extensively studied \citep{reznick2000some}. 
If instead one wants to enforce $p(x)\ge0$ only on a subset $\Omega\subset\mathbb{R}^n$ without requiring nonnegativity on all of $\mathbb{R}^n$, this is possible if $\Omega$ is semialgebraic, meaning it can be specified by a finite number of polynomial inequalities or equalities. Let $\mathbb{R}[x]$ and $\mathbb{R}^n[x]$ denote the sets of real polynomials of $x$ taking values in $\mathbb{R}$ and $\mathbb{R}^n$, respectively. Then any semialgebraic $\Omega$ can be specified as
\begin{equation}
\label{eq:semialg set}
\Omega = \left\{ x\in\mathbb{R}^n~|~g_i(x)\ge0\text{ for }i=1,\ldots,I,~h_j(x)=0\text{ for }j=1,\ldots,J \right\},
\end{equation}
where all $g_i,h_j\in\mathbb{R}[x]$. To guarantee nonnegativity on $\Omega$, we employ a set of polynomials $\Sigma^\Omega$ called the quadratic module generated by the $g_i$ and $h_j$ polynomials. This quadratic module is defined as
\begin{equation}
\label{eq:quad module}
\Sigma^\Omega = \left\{ \sigma_0+\sum\limits_{i=1}^I \sigma_ig_i+\sum\limits_{j=1}^{J}\rho_jh_j~|~\sigma_i\in\Sigma\text{ for }i=0,\ldots,I,~\rho_j\in\mathbb{R}[x]\text{ for }j=1,\ldots,J\right\}.
\end{equation}
Any polynomial in $\Sigma^\Omega$ must be nonnegative on $\Omega$ because $\sigma_0(x)\ge0$ at every $x\in\mathbb{R}^n$, while each $\sigma_i(x)g_i(x)\ge0$ and each $\rho_j(x)h_j(x)=0$  at every $x\in\Omega$. When $\Omega=\mathbb{R}^n$, there are no $g_i$ or $h_j$, so $\Sigma^\Omega=\Sigma$.

In the present case where all functions are polynomial, the pointwise nonnegativity constraint in~\cref{eq:Ba def} can be replaced by the stronger constraint of membership in $\Sigma^\Omega$. This straightforward SOS relaxation of~\cref{eq:Ba def} yields~\cref{eq:Bsos def}, in the second part of \cref{thm:general} below. To state the proposition, we let $\mathcal{P}$ denote the set of all possible orbit periods,
\begin{equation}
\label{eq: P set}
\mathcal{P}=\{ T>0~|~\text{there exists a non-stationary $T$-periodic ODE solution $x(t)$} \}.
\end{equation}
\Cref{thm:general} is aimed at bounding all periods in $\mathcal P$, rather than restricting to POs with certain symmetry. In order for the underlying \cref{thm:main} to apply to all POs, we choose the basis for $\Phi$ to be of the form $a(x) =\ld w(x)$ for some $w$ because such vectors have zero mean over all POs. We ensure that $a(x)$ does not vanish identically over any PO by including $f(x) = \ld x$ as $n$ of the components of $a$; these components can all vanish only at stationary points. Applying \cref{thm:main} in this context leads to \cref{thm:general}.
\begin{prop}
\label{thm:general}
Consider $\ddt{}x(t) = f(x(t))$ with $f\in C^1(\Omega,\mathbb{R}^n)$ and $\Omega\subset\mathbb{R}^n$ forward invariant. Let $w\in C^2(\Omega,\mathbb{R}^m)$ and $a(x) = \ld w(x)$ with $m\ge n$, where the first $n$ components of $w$ are $w_i=x_i$.
\begin{enumerate}[leftmargin=*]
\item The set of all orbit periods is bounded below by
    \begin{equation}
        \inf_\mathcal{P} T \ge \frac{2\pi}{\sqrt{B^a}},
        \label{eq:general result}
    \end{equation}
where
\begin{equation}
    B^a = 
     \inf_{\substack{\mathsf{Q} \succ0 \\V \in C^1(\Omega)}} B \quad\text{s.t.}\quad B a^\mathsf{T}\mathsf{Q} a - \left(\ld a\right)^\mathsf{T} \mathsf{Q} \left(\ld a\right)  + \ld V \ge 0 \quad\forall x \in\Omega.
     \label{eq:Ba def}
\end{equation}

\item Suppose further that $f\in \mathbb{R}^n[x]$, $w\in \mathbb{R}^m[x]$,  and $\Omega$ has the semialgebraic specification~\cref{eq:semialg set} with quadratic module $\Sigma^\Omega$. The set of all orbit periods is bounded below by
    \begin{equation}
        \inf_\mathcal{P} T \ge \frac{2\pi}{\sqrt{B^a_\Sigma}},
        \label{eq:general result SOS}
    \end{equation}
where
\begin{equation}
B^a_\Sigma = 
     \inf_{\substack{\mathsf{Q} \succ0 \\V \in \mathbb{R}[x]}} B \quad\text{s.t.}\quad B a^\mathsf{T}\mathsf{Q} a - \left(\ld a\right)^\mathsf{T} \mathsf{Q} \left(\ld a\right)  + \ld V \in \Sigma^\Omega.
     \label{eq:Bsos def}
    \end{equation}
\end{enumerate}
\end{prop}
\begin{proof}
For the first part, consider a PO of period $T$, with an associated non-stationary trajectory $x(t)$. Let $(B,\mathsf{Q},V)$ satisfy the constraint of \cref{eq:Ba def}. It is enough to show that $T\ge 2\pi/\sqrt{B}$, from which relation the claim follows by taking the infimum over $T$ and the supremum over $B$.

Let $\mathsf{U}\in\mathbb R^{m\times m}$ be a nonsingular matrix satisfying $\mathsf{Q}=\mathsf{U}^\mathsf{T}\mathsf{U}$, which is possible since $\mathsf{Q}\succ0$. Define $\Phi(x) = \mathsf{U}a(x)$, in which case $\Phi\in C^1(\Omega,\mathbb{R}^m)$. The desired lower bound on $T$ follows from \cref{thm:main} once we verify that $(\Phi,V,B)$ satisfy that proposition's conditions. First, the condition~\cref{eq:inequality} follows from the constraint of ~\cref{eq:Ba def} after noting that
\begin{equation}
    |\Phi(x)|^2 = a(x)^\mathsf{T}\mathsf{Q}a(x) \quad\text{and}\quad |\ld\Phi(x)|^2 = \left(\ld a(x)\right)^\mathsf{T}\mathsf{Q}(\ld a(x)).
\end{equation}
Second, $\Phi$ satisfies the mean-zero constraint on all POs since
\begin{equation}
    \int_0^T\Phi(x(t))\,\mathrm{d}t= \int_0^T\mathsf{U} \ld w(x(t))\,\mathrm{d}t = \mathsf{U}\int_0^T \ddt{} w(x(t))\,\mathrm{d}t = 0.
\end{equation}
Finally, to see that $\Phi(x(t))$ cannot vanish identically, suppose it does. This implies that $a(x(t))$ vanishes identically, since $\mathsf{U}$ is nonsingular. The first $n$ components of $a$ are $f=\ld x$, so $a(x(t))$ can vanish only if $x(t)$ is stationary, which contradicts our assumptions. Therefore, all conditions of \cref{thm:main} hold, and the first part is proved.

In the case where $f$ is polynomial, and $V$ and $w$ are chosen to be polynomial, the second part of the proposition follows from the first part by strengthening the polynomial nonnegativity constraint in~\cref{eq:Ba def} into the SOS constraint in~\cref{eq:Bsos def}. The stronger constraint guarantees $B^a\le B^a_\Sigma$, and so $1/\sqrt{B^a}\ge 1/\sqrt{B^a_\Sigma}$.
\end{proof}

The minimization~\cref{eq:Bsos def}, which gives an upper bound on orbit period via~\cref{eq:general result SOS}, is still not quite tractable. For a computable problem, we strengthen the constraint in~\cref{eq:Bsos def} to require membership in a finite-dimensional truncation of the set $\Sigma^\Omega$. This is done by choosing finite-dimensional spaces from which to seek the polynomials $\sigma_i$ and $\rho_j$ that represent polynomials in $\Sigma^\Omega$ as in~\cref{eq:quad module}. Specifically, we choose polynomial basis vectors $b_i$ such that each $\sigma_i = b_i^\mathsf{T} \mathsf{P}_i b_i$, which is made SOS by requiring $\mathsf{P}_i\succeq 0$, and we choose basis vectors $c_j$ such that each $\rho_j=r_j^\mathsf{T}c_j$ for some real vector $r_j$. We also choose a finite basis $c$ for $V$, so that $V=v^\mathsf{T}c$. The constraint in \cref{eq:Bsos def} then becomes
\begin{equation}
    \label{eq:sdp equality}
    B a^\mathsf{T}\mathsf{Q} a - \left(\ld^2 a\right)^\mathsf{T} \mathsf{Q} \left(\ld^2 a\right)  + v^\mathsf{T}\ld c 
= b_0^\mathsf{T} \mathsf{P}_0 b_0 + \sum_{i=1}^I b_i^\mathsf{T} \mathsf{P}_i b_i g_i + \sum_{j=1}^J r_j^\mathsf{T}c_j h_j.
\end{equation}
By expanding the polynomials on both sides of \cref{eq:sdp equality} in the same basis and matching coefficients at any fixed $B$, one obtains a linear relation between $\mathsf Q$, each $\mathsf{P}_i$, $v$ and each $r_j$. This relation implies the SOS constraint in~\cref{eq:Bsos def}, and in general it is a stronger constraint since we have chosen a truncation of $\Sigma^\Omega$. With the constraint in~\cref{eq:Bsos def} strengthened, we obtain
\begin{equation}
\label{eq:sos optimisation problem}
B_\Sigma^a\le 
\inf_{\substack{\mathsf{Q}\in \mathbb{S}^{m}~\\
\mathsf{P}_i\in\mathbb{S}^{m_i}~\\
v\in\mathbb{R}^{n_0}\\
r_j\in\mathbb{R}^{n_j}~}} B\quad s.t.\quad
\begin{array}[t]{l}
(\mathsf Q,\mathsf{P}_i,v,r_j)\text{ satisfy the linear equality implied by }\cref{eq:sdp equality}\\
\mathsf{Q}\succ 0\text{ and all }\mathsf{P}_i\succeq0.
\end{array}
\end{equation}
Here $\mathbb{S}^m$ denotes the set of $m\times m$ symmetric matrices, $m$ is the size of the basis $a$, $m_i$ is the size of $b_i$, $n_0$ is the size of $c$, and $n_j$ is the size of $c_j$. Enlarging the bases $(b_i,c,c_j)$ amounts to choosing a higher-dimensional truncation of $\Sigma^\Omega$ when strengthening the SOS constraint in~\cref{eq:Bsos def}. As this truncation approaches $\Sigma^\Omega$ itself, the right-hand side of~\cref{eq:sos optimisation problem} converges to~$B^a_\Sigma$.

The right-hand side of~\cref{eq:sos optimisation problem} is nearly an SDP. In an SDP, matrices can be constrained to be semidefinite, and equality constraints are permitted if they are affine in the variables being optimized. To obtain an SDP, we enforce $\mathsf Q\succeq0$ rather than strict definiteness, and verify \emph{a posteriori} that $\mathsf{Q}\succ0$. (Alternatively, one could impose $\mathsf{Q}\succeq \epsilon I$ in the SDP for small fixed $\varepsilon$, but we find better results imposing $\mathsf{Q}\succeq0$ to allow arbitrarily small positive eigenvalues.) We also must fix $B$ since the equality constraint in~\cref{eq:sos optimisation problem} is not jointly affine in $B$ and $\mathsf Q$. Thus, rather than solving an SDP optimization problem to minimize $B$, we fix $B$ and solve an SDP feasibility problem: we confirm that the constraint set is non-empty by finding $(\mathsf Q,\mathsf{P}_i,v,r_j)$ that satisfy the constraints. This SDP feasibility computation is then repeated at multiple $B$ values to approximate the desired infimum in \cref{eq:sos optimisation problem}. The smallest $B$ at which one can confirm the SDP's feasibility gives the best lower bound $T\ge 2\pi/\sqrt{B}$ on orbit period. 

Standard numerical methods for finding feasible SDP solutions do not rigorously confirm feasibility due to rounding error. Numerical SDP solutions will slightly violate the equality constraints. Furthermore, the matrices that are supposed to be positive semidefinite can sometimes have slightly negative eigenvalues. This happens especially when $B$ is close to its optimal value, which separates the cases where the constraint set is empty or not -- i.e., where the SDP is infeasible or feasible. A near-optimal value of $B$ computed by the previous paragraph's procedure may still suffice in many applications, despite the rounding error. However, one can go further and confirm feasibility of the SDP at a given $B$ value using validated numerics. \Cref{alg:general} is a procedure for validating a $B$ value that is as close as possible to the exact infimum in~\cref{eq:sos optimisation problem}. The validation strategy is similar to that of \citet{parker2024lorenz}, except it uses rational arithmetic instead of interval arithmetic because, in our implementations, validation by rational arithmetic was significantly faster. The idea behind the algorithm is to project the approximate solution returned by the SDP solver to a rational solution for which the equality constraint in \cref{eq:sdp equality} is satisfied exactly, and which is as close as possible in Euclidean distance to the approximate solution. We then verify that the resulting matrices are truly positive definite. This is close in spirit to the approach of \citet{peyrl2008computing}, but our method computes a full exact kernel for the linear equality relation \cref{eq:sdp equality} and rationalizes the floating point result within this kernel. A similar idea is used in the more sophisticated method of \citet{dostert2021exact}. The result of our algorithm is a computer-assisted proof of a lower bound on orbit period. All of our computational examples in \cref{sec:HH} use this validated algorithm or variant that is modified for symmetric orbits using the ideas of the next subsection.

\begin{algorithm}[p]
\caption{Verified computation of lower bounds on periods using rational arithmetic.}
\label{alg:general}
    \vspace{6pt}Inputs: $f\in\mathbb{R}^n[x]$, $a\in\mathbb{R}^m[x]$, $c\in\mathbb{R}^{n_0}[x]$, $b_i \in \mathbb{R}^{m_i}[x]$ for $i=0,\dots,I$, and $c_j\in\mathbb{R}^{n_j}[x]$ for $j=1,\dots,J$, all with integer coefficients.\\[6pt]
    Output: $B\in\mathbb{Q}$ such that $\inf_\mathcal{P} T \ge \frac{2\pi}{\sqrt{B}}$.
    \begin{enumerate}
    \item Fix a first value of $B\in\mathbb{Q}$.
    \item Attempt to solve the SDP feasibility problem~\cref{eq:sos optimisation problem}, using floating point arithmetic and with the constraint $tr\,\mathsf{Q}=1$ included to fix the variables' magnitudes. 
    \item If the SDP solution fails to converge, or if it returns a solution that does not satisfy the constraints within a specified tolerance, increase $B$ and repeat step 2. If the SDP solution returns floating point variables $\mathsf{Q}\in\mathbb{S}^m$, $\mathsf{P}_i\in\mathbb{S}^{m_i}$, $v\in\mathbb{R}^{n_0}$ and $r_j\in\mathbb{R}^{n_j}$ that nearly satisfy the constraints, including all matrices being positive definite, proceed with rational validation.
    \item Express the equality constraint in~\cref{eq:sos optimisation problem} as a linear relation $\mathsf{A} y = 0$, where the vector $y$ contains all variables in $\mathsf{Q}$, $\mathsf{P}_i$, $v$ and $r_j$. Choose $\mathsf A$ to have integer entries, which may require multiplying all terms by the denominator of $B$.
    \item Use rational Gaussian elimination to construct a matrix $\mathsf{K}$ whose columns are an exact basis for the kernel of $\mathsf A$.
    \item Using the floating point result from step 2, construct a rational vector $y_0$ that approximately solves $\mathsf{A}y=0$.
    \item Find a floating point vector $x_0$ minimizing $\left|\mathsf{K}x_0 - y_0\right|$ using the LSQR algorithm.
    \item Approximate $x_0$ as a rational vector $x$, so that $y=\mathsf K x$ exactly solves $\mathsf Ay=0$.
    \item Build rational $\mathsf{Q}$, $\mathsf{P}_i$, $v$ and $r_j$ from $y$ that satisfy the equality constraint in~\cref{eq:sos optimisation problem} by construction.
    \item Validate that the rational matrices satisfy $\mathsf{Q}\succ0$ and $\mathsf{P}_i\succeq0$. This is done via $\mathsf{LDL}^\mathsf{T}$ decomposition, which is exact with rational arithmetic.
    \item If validation fails at any step, increase $B$ and return to step 2. If validation succeeds, decrease $B$ and return to step 2. Carry out a line search for the smallest $B$ value that can be validated, terminating when reaching the desired precision in $B$.
\end{enumerate}
\end{algorithm}

\subsection{Systems with symmetries}
\label{sec: sym}

Many systems of practical interest have symmetries. An ODE system \cref{eq:system} is said to be symmetric under a linear transformation $\Lambda:\mathbb{R}^n\to\mathbb{R}^n$ if the domain is $\Lambda$-invariant, meaning $\Lambda(\Omega)=\Omega$, and the vector field is $\Lambda$-equivariant, meaning $f(\Lambda x) = \Lambda f(x)$ for all $x\in\Omega$. Such ODEs will have a trajectory $x(t)$ if and only if $\Lambda x(t)$ is also a trajectory. This can happen by the trajectory itself being $\Lambda$-invariant, or by there being a set of trajectories that each break the symmetry but map to each other under it.

Symmetry is useful to us in two ways. \Cref{sec: block diag} describes how symmetry allows matrices in the SDPs to be block diagonalized, which reduces cost and improves numerical accuracy. \Cref{sec: sym POs} describes how to compute bounds that apply only to POs that are $\Lambda$-invariant under a chosen $\Lambda$.

For simplicity, we restrict our attention to the common case of a sign symmetry, which negates a set of coordinates. In this case, $\Lambda$ is a \textit{signature matrix}, meaning it is a diagonal matrix with each diagonal entry being 1 or $-1$. To avoid the trivial case we assume that $\Lambda$ is not the identity matrix, so it changes at least one sign. Every signature matrix $\Lambda$ induces a class of \textit{even} functions $r:\Omega\to\mathbb{R}$ that satisfy $r(\Lambda x)=r(x)$ for all $x\in\Omega$, and a class of \textit{odd} functions that satisfy $r(\Lambda x)=-r(x)$. The theorems of this subsection could be extended to other finite groups of orthogonal symmetries following ideas in \citet{gatermann2004symmetry} and the appendix of \citet{oeri2023convex}.

\subsubsection{Block diagonalization in SDPs}
\label{sec: block diag}

We use a standard type of symmetrization argument to show that the matrices and polynomials appearing in the minimization in~\cref{thm:general} can be restricted to symmetric subspaces. \Cref{thm:sym reduce} justifies that the polynomial $V$ is invariant under the symmetry, and that the matrix $\mathsf Q$ can be block diagonalized into two blocks $\mathsf Q^e$ and $\mathsf Q^o$, which is denoted as $\mathsf Q = \mathsf Q^e \oplus \mathsf Q^o$. We then explain how all matrices in the resulting SDPs split into two blocks.

\begin{prop}
\label{thm:sym reduce}
Consider an ODE on $\Omega$ as in \cref{thm:general}, and suppose there is a signature matrix $\Lambda$ under which $\Omega$ is invariant and $f$ is equivariant. Let $w\in C^2(\Omega,\mathbb{R}^m)$ be of the form $w=(w^e,w^o)$, where all components of $w^e$ or $w^o$ are even or odd under $\Lambda$, respectively. Let the $n$ components of $x$ be among the $m\ge n$ components of $w$ also, and define $a=(a^e,a^o) =(\ld w^e,\ld w^o)$.
\begin{enumerate}[leftmargin=*]
\item The infimum $B^a$ defined by~\cref{eq:Ba def}, which bounds orbit periods by~\cref{eq:general result}, is unchanged if the minimization is additionally constrained such that $V$ is $\Lambda$-invariant and $\mathsf Q$ is block diagonalized such that $a^{\mathsf T}\mathsf{Q}a=(a^e)^{\mathsf T}\mathsf{Q}^ea^e+(a^o)^{\mathsf T}\mathsf{Q}^oa^o$.
\item Suppose further that $f\in \mathbb{R}^n[x]$, $w\in\mathbb{R}^m[x]$, and $\Omega$ has semialgebraic specification \cref{eq:semialg set} with all $g_i$ and $h_j$ $\Lambda$-invariant. Then, the infimum $B^a_\Sigma$ defined by~\cref{eq:Bsos def}, which bounds orbit periods by~\cref{eq:general result SOS}, is unchanged if the minimization is additionally constrained as in part 1 and by all $\sigma_i$ and $\rho_j$ being $\Lambda$-invariant in the quadratic module representation~\cref{eq:quad module}.
\end{enumerate}
\end{prop}
\begin{proof}
    For the first part, fix any $(w^e,w^o)$ satisfying the hypotheses, which fixes $(a^e,a^o)$. Let $(B,\mathsf{Q},V)$ satisfy the constraints of \cref{eq:Ba def}. It suffices to show that there exist $(\widehat{\mathsf Q},\widehat{V})$ satisfying the same constraints with the same $B$, where $\widehat{\mathsf Q}$ is block diagonal and $\widehat V$ is $\Lambda$-invariant. Such $\widehat{\mathsf Q}$ and $\widehat{V}$ can be constructed from $\mathsf Q$ and $V$ by averaging the inequality in~\cref{eq:Ba def} at $x$ and at~$\Lambda x$.

By assumption, the even and odd parts of $w$ satisfy $w^e(\Lambda x) = w^e(x)$ and $w^o(\Lambda x)=-w^o(x)$. The Lie derivative $a^e = \ld w^e$ is also even since
    \begin{equation}
    \begin{aligned}
        a^e(\Lambda x) &= \ld w^e (\Lambda x) 
        \\&= f(\Lambda x) \cdot D w^e (\Lambda x) 
        \\&= \left(\Lambda f(x)\right) \cdot \Lambda^\mathsf{-T} D (w^e\circ \Lambda) (x)
        \\&= f(x)\cdot D (w^e \circ \Lambda)(x)
        \\&= f(x) \cdot D w^e (x)
        \\&= a^e(x),
        \end{aligned}
    \end{equation}
where the gradient $D$ is always with respect to the entire argument. A similar calculation shows that $a^o=\ld w^o$ is odd. Denote the blocks of $\mathsf Q$ as
    \begin{equation}
        \mathsf{Q} = \begin{bmatrix}
            \mathsf{Q}^e &\mathsf{Q}^c\\(\mathsf{Q}^c)^\mathsf{T}&\mathsf{Q}^o
        \end{bmatrix},
    \end{equation}
where the square blocks $\mathsf{Q}^e$ and $\mathsf{Q}^o$ match the sizes of the vectors $a^e$ and $a^o$. Observe that
\begin{equation}
a(x)^\mathsf{T}\mathsf{Q} a(x) + a(\Lambda x)^\mathsf{T}\mathsf{Q} a(\Lambda x) = 2a^e(x)^\mathsf{T}\mathsf{Q}^e a^e(x) + 2 a^o(x)^\mathsf{T}\mathsf{Q}^o a^o(x),
\end{equation}
and likewise for $(\ld a)^\mathsf{T}\ld\mathsf{Q} a$. Since the inequality between $(B,\mathsf{Q},V)$ in~\cref{eq:Ba def} holds at all $x\in\Omega$, it also holds when evaluated at $\Lambda x$ for all $x\in\Omega$. Adding these two inequalities and dividing by 2 gives
\begin{equation}
\label{eq:Ba cons sym}
B \left[(a^e)^\mathsf{T}\mathsf{Q}^e a^e+(a^o)^\mathsf{T}\mathsf{Q}^o a^o\right] - \left[\left(\ld a^e\right)^\mathsf{T} \mathsf{Q}^e\left(\ld a^e\right)+\left(\ld a^o\right)^\mathsf{T} \mathsf{Q}^o\left(\ld a^o\right)\right]  + \ld \widehat V \ge 0
    \end{equation}
where $\widehat V(x) = \frac{1}{2}[V(x) + V(\Lambda x)]$ is $\Lambda$-invariant. Note that~\cref{eq:Ba cons sym} is exactly the constraint in~\cref{eq:Ba def}, for $\widehat{\mathsf Q} = \mathsf Q^e \oplus \mathsf Q^o$, with $\mathsf Q^e\succ0$ and $\mathsf Q^o\succ0$, and $\Lambda$-invariant $\widehat V$, so the first part of the proposition is proved.

For the second part, fix any $(w^e,w^o)$ satisfying the hypotheses, and let $(B,\mathsf{Q},V)$ satisfy the constraints of~\cref{eq:Bsos def}. The constraint requiring membership in the quadratic module $\Sigma^\Omega$ means there exist $\sigma_i\in\Sigma$ and $\rho_j\in\mathbb R[x]$ such that
\begin{equation}
\label{eq:Bsos cons sym}
B a^\mathsf{T}\mathsf{Q} a - \left(\ld a\right)^\mathsf{T} \mathsf{Q} \left(\ld a\right)  + \ld V = \sigma_0+\sum\limits_{i=1}^I \sigma_ig_i+\sum\limits_{j=1}^{J}\rho_jh_j
\end{equation}
for all $x\in\mathbb R^n$. It suffices to show that there exist $(\widehat{\mathsf{Q}},\widehat{V},\widehat{\sigma}_i,\widehat{\rho}_j)$ satisfying the constraints of~\cref{eq:Bsos def} with the same $B$, where $\widehat{\mathsf{Q}}$ is block diagonal and the polynomials are all $\Lambda$-invariant. This follows from a symmetrization argument as in the first part. The equality constraint~\cref{eq:Bsos cons sym} holds when all functions are evaluated at $x$ or at $\Lambda x$. Averaging these two versions of the equality give the analogous equality relating $(B,\widehat{\mathsf{Q}},\widehat{V},\widehat{\sigma}_i,\widehat{\rho}_j)$, where $\widehat{\mathsf{Q}}$ and $\widehat{V}$ are as in the first part, and $\widehat\sigma_i$ and $\widehat\rho_j$ are symmetrizations defined analogously to $\widehat V$. The symmetrized $\widehat{\sigma}_i$ are SOS since the $\sigma_i$ are, which completes the proof of the second part.
\end{proof}

Recall that $B^a_\Sigma$ defined by~\cref{eq:general result SOS} can be bounded above in practice by choosing a finite-dimensional truncation of the quadratic module $\Sigma^\Omega$ and converting the problem to an SDP. As described below in \cref{thm:general}, the truncated quadratic module constraint is enforced by choosing bases $b_i$, $c_j$, and $c$ and requiring there to exist vectors $v$, $r_j$ and matrices $\mathsf P_i\succeq0$ such that the equality \cref{eq:sdp equality} holds. In the present case, where the ODE is symmetric under $\Lambda$ and we choose $w=(w^e,w^o)$, the second part of \cref{thm:sym reduce} lets terms in the equality~\cref{eq:sdp equality} be further constrained without affecting the feasibility of the resulting SDP. In particular, the matrix $\mathsf Q$ in the first the two terms of~\cref{eq:sdp equality} can be block diagonal with blocks $\mathsf Q^e$ and $\mathsf Q^o$, and the other terms in~\cref{eq:sdp equality} represent $\Lambda$-invariant polynomials. This $\Lambda$-invariance lets us choose the bases $c$ and $c_j$ to be $\Lambda$-invariant (i.e., even). For the SOS polynomials $b_i^\mathsf{T} \mathsf{P}_i b_i$, the bases can be chosen with even and odd parts as in $b_i=(b_i^e,b_i^o)$, and the matrices can be correspondingly block diagonalized such that $b_i^\mathsf{T} \mathsf{P}_i b_i=(b_i^e)^\mathsf{T} \mathsf{P}_i^e b_i^e+(b_i^o)^\mathsf{T} \mathsf{P}_i^o b_i^o$. This block diagonalization sacrifices nothing; the polynomials are SOS if and only if they can be represented by blocks $\mathsf P_i^e,\mathsf P_i^o\succeq0$ \citep{gatermann2004symmetry}. In summary, sign symmetry lets us decompose the equality constraint~\cref{eq:sdp equality} as
\begin{multline}
\label{eq:sdp equality exploiting sym}
B \left[(a^e)^\mathsf{T}\mathsf{Q}^e a^e+(a^o)^\mathsf{T}\mathsf{Q}^o a^o\right] - \left[\left(\ld a^e\right)^\mathsf{T} \mathsf{Q}^e\left(\ld a^e\right)+\left(\ld a^o\right)^\mathsf{T} \mathsf{Q}^o\left(\ld a^o\right)\right]  + v^\mathsf{T}\ld c \\= (b^e_0)^\mathsf{T} \mathsf{P}^e_0 b^e_0 + (b^o_0)^\mathsf{T} \mathsf{P}^o_0 b^o_0 + \sum_{i=1}^I \left[(b^e_i)^\mathsf{T} \mathsf{P}^e_i b^e_i + (b^o_i)^\mathsf{T} \mathsf{P}^o_i b^o_i\right] g_i + \sum_{j=1}^J r_j^\mathsf{T}c_j h_j.
\end{multline}
The minimization~\cref{eq:sos optimisation problem} subject to this constraint then becomes
\begin{equation}
\label{eq:sos optimisation problem exploiting sym}
B_\Sigma^a\le 
\inf_{\substack{\mathsf{Q}^e\oplus\mathsf{Q}^o\in\mathbb S^m,~\\
~\mathsf{P}_i^e\oplus\mathsf{P}_i^o\in\mathbb S^{m_i},~\\
v\in\mathbb{R}^{n_0}\\
r_j\in\mathbb{R}^{n_j}~}} B\quad s.t.\quad
\begin{array}[t]{l}
(\mathsf Q^e,\mathsf Q^o,\mathsf{P}^e_i,\mathsf{P}^o_i,v,r_j)\text{ satisfy the linear equality implied by }\cref{eq:sdp equality exploiting sym}\\
\mathsf{Q}^e\succ 0,\,\mathsf{Q}^o\succ 0,\text{ and all }\mathsf{P}^e_i\succeq0 \text{ and } \mathsf{P}^o_i\succeq0.
\end{array}
\end{equation}

For any choice of finite even bases $a^e,b^e_i,c,c_j$ and odd bases $a^o,b^o_i$, the right-hand infimum in~\cref{eq:sos optimisation problem exploiting sym} coincides with the right-hand infimum in~\cref{eq:sos optimisation problem} if $a$ and $b_i$ have the same spans as $(a^e,a^o)$ and $(b_i^e,b_i^o)$, respectively. With $B$ fixed, and positive definiteness relaxed to semidefiniteness, the resulting SDP feasibility problems from~\cref{eq:sos optimisation problem,eq:sos optimisation problem exploiting sym} are either both feasible or both infeasible. However, the symmetry-reduced version has significantly lower computational cost because each matrix variable is roughly half as large. It is also straightforward to adapt \cref{alg:general} for validated numerics so that the symmetry-reduced SDP from~\cref{eq:sos optimisation problem exploiting sym} is used rather than the SDP from~\cref{eq:sos optimisation problem}. This modification of \cref{alg:general} is used in our provided code, which produces the validated numerics for the examples of \cref{sec:HH,sec:lorenz}.

\subsubsection{Bounds for symmetric orbits}
\label{sec: sym POs}

We now turn to bounds that apply to all $\Lambda$-invariant POs but not necessarily to other POs. If a PO is $\Lambda$-invariant, it may or may not lie in the subset of $\Omega$ that is fixed by $\Lambda$, which is
\begin{equation}
\Omega_\Lambda=\{x\in\Omega\,|\,\Lambda x = x\}.
\end{equation}
For POs that are contained in $\Omega_\Lambda$, one should be able to formulate a lower-dimensional ODE for the dynamics on $\Omega_\Lambda$, and then use our general approach to bound their periods. We thus focus on the other case of POs that are $\Lambda$-invariant but not contained in $\Omega_\Lambda$. 
We let $\mathcal{P}_\Lambda$ contain all periods of such orbits:
\begin{equation}
\label{eq: P set Lambda}
\mathcal{P}_\Lambda=\{ T\in\mathcal{P}~|~\text{there is a $T$-periodic orbit that is $\Lambda$-invariant but not in $\Omega_\Lambda$}\}.
\end{equation}
With minor modification, the optimization problems in \cref{thm:general} can give bounds applying to orbit periods in $\mathcal P_\Lambda$, rather than in the set $\mathcal P$ of all periods. The only change needed is for the basis $a(x) = a^o(x)$ to be purely odd under $\Lambda$, rather than being the Lie derivative of another basis $w$. Then, $\Phi$ will have zero mean on all $\Lambda$-invariant POs, but not generally on other POs. This observation gives the following proposition, in which we consider only sign symmetries for simplicity of exposition.

\begin{prop}
\label{thm:symmetric POs}
Consider an ODE on $\Omega$ as in \cref{thm:general}, and suppose there is a signature matrix $\Lambda$ under which $\Omega$ is invariant and $f$ is equivariant. Let $a^o\in C^1(\Omega,\mathbb{R}^m)$ be odd under $\Lambda$ with $a^o(x)=0$ only if $x\in\Omega_\Lambda$.
\begin{enumerate}[leftmargin=*]
\item The infimum $B^a$ defined by~\cref{eq:Ba def} with $a=a^o$ bounds symmetric orbit periods below by
\begin{equation}
        \inf_{\mathcal{P}_\Lambda} T \ge \frac{2\pi}{\sqrt{B^a}},
        \label{eq:symmetric result}
    \end{equation}
and minimization in~\cref{eq:Ba def} has the same infimum if $V$ is constrained to be $\Lambda$-invariant.
\item Suppose further that $f\in \mathbb{R}^n[x]$, $a^o\in\mathbb{R}^m[x]$, and $\Omega$ has semialgebraic specification \cref{eq:semialg set} with all $g_i$ and $h_j$ $\Lambda$-invariant. The infimum $B^a_\Sigma$ defined by~\cref{eq:Bsos def} with $a=a^o$ bounds symmetric orbit periods below by
\begin{equation}
        \inf_{\mathcal{P}_\Lambda} T \ge \frac{2\pi}{\sqrt{B^a_\Sigma}},
        \label{eq:symmetric sos result}
    \end{equation}
\end{enumerate}
and the minimization in~\cref{eq:Bsos def} has the same infimum if $V$ is $\Lambda$-invariant and so are all $\sigma_i$ and $\rho_j$ in the quadratic module representation~\cref{eq:quad module}.
\end{prop}
\begin{proof}
The first part is proved similarly to the first part of \cref{thm:general}. Let the trajectory $x(t)$ be a $\Lambda$-invariant PO of fundamental period $T$ that is not contained in $\Omega_\Lambda$, and let $(B,\mathsf{Q},V)$ satisfy the constraint of \cref{eq:Ba def}. It suffices to show that $T\ge 2\pi/\sqrt{B}$.

Since the PO is $\Lambda$-invariant but not in $\Omega_\Lambda$, for all $t$ it holds that $x(t+T/2)=\Lambda x(t)$. Therefore $a^o(x)$ and $\Phi(x)$ are mean-zero over the orbit since
\begin{equation}
\int_0^T a^o(x(t))\,\mathrm{d}t = \int_0^T a^o(\Lambda x(t-T/2))\,\mathrm{d}t = \int_0^T -a^o(x(t-T/2))\,\mathrm{d}t = -\int_0^T a^o(x(t))\,\mathrm{d}t = 0,
\end{equation}
and
\begin{equation}
    \int_0^T\Phi(x(t))\,\mathrm{d}t= \mathsf{U} \int_0^T a^o(x(t))\,\mathrm{d}t = 0,
\end{equation}
recalling that $\Phi(x) = \mathsf{U}a^o(x)$ where $\mathsf{Q}=\mathsf{U}^\mathsf{T}\mathsf{U}$. Meanwhile, $\Phi(x)$ does not vanish everywhere on $x(t)$; it can vanish only where $a^o(x)$ does since $\mathsf Q\succ0$, and $a^o(x)=0$ only in $\Omega_\Lambda$, which does not contain $x(t)$ by assumption. Since $\Phi$ is mean-zero on $x(t)$ but does not vanish pointwise, \cref{thm:main} gives $T\geq 2\pi/\sqrt{B}$ as desired. Lastly, the fact that the minimization defining $B^a$ has the same infimum over only $\Lambda$-invariant $V$ follows from the same symmetrization argument used to prove~\cref{thm:sym reduce}.

For the second part of the proposition, the lower bound~\cref{eq:symmetric sos result} follows from~\cref{eq:symmetric result} because the minimization defining $B^a_\Sigma$ is an SOS relaxation of the one defining $B^a$ in the polynomial case, and so $B^a\le B^a_\Sigma$. The justification of minimizing only over $\Lambda$-invariant polynomials is again as in~\cref{thm:sym reduce}.
\end{proof}

\begin{rmk}
\label{rmk: sym SDP}
Lower bounds on periods of symmetric POs can be computed via~\cref{eq:symmetric sos result} by formulating SDPs that bound $B^a_\Sigma$ above. The $\Lambda$-invariance of polynomials justified by \cref{thm:symmetric POs} allows matrices in the SDPs to be block diagonalized. The resulting SDPs take the same form as in~\cref{eq:sos optimisation problem exploiting sym}, with $\mathsf Q^o\succeq0$ relaxed as usual but with no $\mathsf Q^e$ block at all in this case. With straightforward modification of \cref{alg:general} one can verify a value of $B \ge B^a_\Sigma$.
\end{rmk}

\section{Examples}
\label{sec:examples}

We now demonstrate the methods of \cref{sec:thm} on two classic ODE examples -- the H\'enon--Heiles system, which displays Hamiltonian chaos, and the Lorenz system, which displays dissipative chaos. Both systems have sign symmetries that can be exploited as in \cref{sec: sym}, and we use these symmetries to construct validated numerical bounds that are sharp to five digits, for both general and symmetric POs in each system. The validated computations are variants of \cref{alg:general}, with certain choices of the polynomial bases in each case, which we detail below.

Julia code to reproduce results in this section is available at \url{https://github.com/jeremypparker/ODE_Period_Bounds}. We use the JuMP \citep{JuMP} and DynamicPolynomials.jl \citep{dynamicpolynomials} packages to assemble the SDPs, which are then solved using the Hypatia solver \citep{hypatia} with quadruple precision (128-bit) floating point arithmetic from DoubleFloats.jl \citep{Sarnoff_DoubleFloats_2022}. Rational validation is performed using the QQ type from Nemo.jl \citep{Nemo.jl-2017}. We have not optimized our implementation of \cref{alg:general} but have made some effort to be efficient. Runtimes are given for all computations reported below, measured by repeating each computation several times on a laptop using a single core of an Intel Core i7-13800H. The system had 32GB of memory, but most computations used far less.

\subsection{Minimal periods in the H\'enon--Heiles system}
\label{sec:HH}

The H\'enon--Heiles system  \citep{henon1964applicability} is defined by the Hamiltonian
\begin{equation}
\label{eq:HH}
    H(x) = \frac{1}{2}\left(x_1^2+x_2^2+x_3^2+x_4^2\right) + x_1^2 x_2 - \frac{1}{3}x_2^3,
\end{equation}
which induces an ODE system $\ddt{}x(t)=f(x(t))$ with the polynomial right-hand side
\begin{equation}
\label{eq:hh f}
f(x) = \left(x_3,\,x_4,\,-x_1-2x_1x_2,\,-x_2-x_1^2+x_2^2\right).
\end{equation}
This system has a $2\pi/3$ rotational symmetry that we do not exploit, as well as a sign-change symmetry that we do exploit, where $x$ is transformed by the linear mapping
\begin{equation}
    \Lambda = \begin{bmatrix}
        -1&0&0&0\\
        0&1&0&0\\
        0&0&-1&0\\
        0&0&0&1
    \end{bmatrix}.
    \label{eq:HH D}
\end{equation}

There are fixed points at $x= 0$ and at the three points $x=(0,1,0,0)$ and $x=(\pm\sqrt{3}/2,-1/2,0,0)$ that are related by the rotational symmetry. The nonzero fixed points have energy $H(x)=1/6$, which is a bifurcation value: For $0\le H(x)\le 1/6$, there is a compact region of bounded (generally chaotic) trajectories around the origin, whereas for $H(x)>1/6$ generic trajectories are not bounded. No bounded trajectories exist when $H<0$.

We restrict attention to the energies $H\le1/6$, for which there are an infinitude of POs. In order to find a candidate for the PO with the shortest period in this energy range, we compute the periods of eight orbits described by \citet{churchill1979survey} that bifurcate from the origin at $H=0$. These POs are continued as $H$ is increased, using Newton-Raphson iteration combined with a shooting method. Among the POs we have computed, the shortest period is attained by two orbits that map to each other under the sign symmetry $\Lambda$. \Cref{fig:hh} shows these two orbits, which have been called $\Pi_7$ and $\Pi_8$ in the past~\cite{churchill1979survey}.
The periods of these orbits decrease monotonically as $H$ is raised; the period approaches $2\pi$ as $H\to0$, while it decreases until $T\approx 6.02248$ at $H=1/6$. In \cref{sec:HHgeneral} we prove that this period is indeed minimal, at least to 5 digits, by verified computation of lower bounds for all POs with $H\le1/6$. In \cref{sec:HHsym} we compute still larger lower bounds for POs that are invariant under $\Lambda$, which excludes the $\Pi_7$ and $\Pi_8$ orbits.

\begin{figure}
    \centering
    \begin{tikzpicture}
\begin{groupplot}[
    group style={group size=2 by 1, horizontal sep=2.0cm, vertical sep=1.8cm},
    width=7.2cm,
    height=5.6cm,
    axis equal image,
    grid=both,
    grid style={line width=.1pt, draw=gray!25},
    major grid style={line width=.2pt, draw=gray!45},
    tick label style={font=\small},
    label style={font=\small},
    title style={font=\small},
    xmin=-0.5, xmax=0.5,
    ymin=-0.5, ymax=0.5,
    xtick={-0.4,0,0.4},
    ytick={-0.4,0,0.4},
    minor tick num=0,
    ylabel style={
        rotate=-90,
        xshift=0.4cm
    },
]

\nextgroupplot[
    xlabel={$x_1$},
    ylabel={$x_2$},
]
\addplot+[no marks, thick] table[x index=1, y index=2] {pi7_E_1over6.dat};
\addlegendentry{$\Pi_7$}
\addplot+[no marks, thick, dashed] table[x index=1, y index=2] {pi8_E_1over6.dat};
\addlegendentry{$\Pi_8$}

\legend{}

\nextgroupplot[
    xlabel={$x_3$},
    ylabel={$x_4$},
]
\addplot+[no marks, thick] table[x index=3, y index=4] {pi7_E_1over6.dat};
\addlegendentry{$\Pi_7$}
\addplot+[no marks, thick, dashed] table[x index=3, y index=4] {pi8_E_1over6.dat};
\addlegendentry{$\Pi_8$}

\legend{}

\end{groupplot}
\end{tikzpicture}
\caption{Two POs of the H\'enon--Heiles system \cref{eq:HH}, projected onto the position coordinates (left) and momentum coordinates (right), that have the shortest known period of $T\approx6.02248$ among energies $H\le1/6$. The POs, which have been called $\Pi_7$ ($\color{blue}\solidrule$) and $\Pi_8$ ($\color{red}\dashedrule$),  map to each other by the symmetry~\cref{eq:HH D}.}
\label{fig:hh}
\end{figure}

\subsubsection{Bounds on periods of general POs}
\label{sec:HHgeneral}

To find lower bounds applying to all POs of the H\'enon--Heiles system, we use a variant of \cref{alg:general} applied to \cref{eq:sos optimisation problem exploiting sym}.
In the constraint \cref{eq:sdp equality exploiting sym}, the bases for $\Phi$ are $a^e=\ld w^e$ and $a^o=\ld w^o$, with $w^e$ and $w^o$ being the even and odd monomials in $x$ of degrees $1$ through $d_a-1$.
The domain is the semialgebraic set
\begin{equation}
\label{eq:hh domain}
    \Omega = \{x\in\mathbb{R}^3~|~1-6H(x)\geq0\},
\end{equation}
so that $I=1$, $J=0$ and $g_1=1-6H$ in \cref{eq:quad module}.
In contrast to \citet{oeri2023convex}, we do not restrict the domain to a compact set, nor do we constrain the energy to be nonnegative, so \cref{eq:hh domain} includes unbounded trajectories, but this does not pose a problem for our method.

\Cref{tab:HH bases} summarizes how we have chosen the various polynomial basis vectors appearing in the relation~\cref{eq:sdp equality exploiting sym}, which implies the equality constraint for the SDP~\cref{eq:sdp equality exploiting sym}. Most basis vectors simply contain monomials up to the required degree that are even or odd under the symmetry $\Lambda$. However, the $b_i^e$ and $b_i^o$ vectors must be chosen with care. For robustness of \cref{alg:general}, we want the matrices $\mathsf{P}_i^e$ and $\mathsf{P}_i^o$ to be strictly positive definite, so that the validation step does not start with a numerical matrix whose eigenvalues are close to machine precision. The matrices can be positive definite only if $b_i^e$ and $b_i^o$ vanish at certain fixed points, for the following reasons.

\begin{table}[tp]
    \centering
    \renewcommand{\arraystretch}{1.3}
\begin{tabular}{
  >{\centering\arraybackslash}m{1.2cm}
  >{\centering\arraybackslash}m{2cm}
  >{\centering\arraybackslash}m{3cm}
  >{\centering\arraybackslash}m{6.5cm}
}
  \hline Basis & Maximum degree & Purpose & Terms \\\hline
         $a^e$ & $d_a$ & Defining $\Phi(x)$ & $\ld w^e$, where $w^e$ contains monomials of degrees 1 to $d_a-1$\\
         $a^o$ & $d_a$ & Defining $\Phi(x)$ & $\ld w^o$, where $w^o$ contains monomials of degrees 1 to $d_a-1$\\
         $b_0^e$ & $d_b$ & Enforcing nonnegativity & See text \\
         $b_0^o$ & $d_b$ & Enforcing nonnegativity & See text \\
         $b_1^e$ & $d_b-2$ & Enforcing $H\le 1/6$ & Even monomials of degrees 1 to $d_b-2$ \\
         $b_1^o$ & $d_b-2$ & Enforcing $H\le 1/6$ & Odd monomials of degrees 1 to $d_b-2$ \\
         $c$ & $d_c=2d_b-1$ & Defining $V(x)$ & Even monomials of degrees 1 to $d_c$\\\hline
    \end{tabular}
    \caption{Polynomial basis vectors used in the constraint \cref{eq:sdp equality exploiting sym} to bound periods of general POs in the H\'enon--Heiles system \cref{eq:hh f}. These vectors are starting points for an iterative procedure (see text) that prunes terms with near-zero coefficients and re-solves the SDP~\cref{eq:sdp equality exploiting sym}.}
    \label{tab:HH bases}
\end{table}

At the nonzero fixed points, where $f$ and $g_1$ vanish, the constraint~\cref{eq:sdp equality exploiting sym} reduces to $(b^e_0)^\mathsf{T} \mathsf{P}^e_0 b^e_0 + (b^o_0)^\mathsf{T} \mathsf{P}^o_0 b^o_0=0$. These matrices can be positive definite only if $b_0^e$ and $b_0^o$ vanish at the nonzero fixed points. We thus generate $b_0^e$ using the components of the ODE right-hand side~\cref{eq:hh f}, multiplying $f_1$ and $f_3$ by odd monomials, and multiplying $f_2$ and $f_4$ by even monomials, to produce a set of even terms up to degree $d_b$. This set is then reduced to a linearly independent basis using QR factorization and truncation (see the code), and this even basis forms the $b_0^e$ vector. The odd basis vector $b_0^o$ is generated analogously. At the fixed point on the origin, where $f$ and the $b_0$ vectors vanish, but $g_1$ does not, the constraint~\cref{eq:sdp equality exploiting sym} reduces to $(b^e_1)^\mathsf{T} \mathsf{P}^e_1 b^e_1 + (b^o_1)^\mathsf{T} \mathsf{P}^o_1 b^o_1=0$. These matrices can be positive definite only if their basis vectors vanish at the origin, which we ensure by using all monomials up to degree $d_b$ but excluding the constant term from $b_1^e$.

We augment \cref{alg:general} with one important additional step, as well as minor modifications to allow block diagonalization with respect to the symmetry. The additional step occurs after the SDP is approximately solved using floating point arithmetic. For all the basis vectors that are multiplied on both sides of a matrix, we identify components whose corresponding diagonal matrix entries are less than $10^{-4}$. All of these components are removed from the basis vectors, and then the new, smaller SDP is solved. This is done four times, yielding approximate SDP solutions whose matrices are strictly positive definite with larger minimum eigenvalues. This aids the much slower validation steps that follow. To make sure that \cref{thm:sym reduce} still applies after such pruning, we check that the entries of $a^e$ and $a^o$ still include the even and odd elements of $f$, respectively.

\Cref{tab:HH} reports rigorously validated lower bounds on orbit periods that we computed using several different choices for maximum polynomial degrees. Our best bound is sharp to five digits; it agrees to this precision with the period of the symmetry-related orbits $\Pi_7$ and $\Pi_8$ at $H=1/6$. This bound is computed and validated in about 9 seconds on a single core. Higher polynomial degrees give better bounds, and they generally incur higher computational cost (cf.\ \cref{tab:HH}). For each combination of degrees $(d_a,d_b)$, the degree $d_c$ is chosen so that the polynomial degree of $\ld V$ is equal to the maximum degree of the other terms in \cref{eq:sdp equality exploiting sym}, which means that increasing $d_c$ further cannot improve the bound. One can also choose to have the same maximum degrees in terms containing $\mathsf Q$ and $\mathsf P$, which occurs when $d_a=d_b-1$ in the present example. This is the choice made in the first and last rows of~\cref{tab:HH}. The middle row has $d_a<d_b-1$, which seems to give a worse tradeoff between sharpness and computation time.

\begin{table}[tp]
    \centering
\begin{tabular}{
  >{\centering\arraybackslash}m{.7cm}
  >{\centering\arraybackslash}m{.7cm}
  >{\centering\arraybackslash}m{.7cm}
  >{\centering\arraybackslash}m{1cm}
  >{\centering\arraybackslash}m{2.8cm}
  >{\centering\arraybackslash}m{1.4cm}
  >{\centering\arraybackslash}m{2cm}
}        \hline$d_a$&$d_b$&$d_c$&$B$&\textbf{Lower bound on period}&SDP time (s)&Validation time (s)\\\hline
        2&3&5&1.28572&\textbf{5.5412}&0.33&0.01\\
        2&4&7&1.25210&\textbf{5.6151}&7.95&0.38\\
        3&4&7&1.08846&\textbf{{6.0224}}&8.59&0.49\\\hline
    \end{tabular}
    \caption{Validated lower bounds on orbit periods for the H\'enon--Heiles system, for different maximum degrees $(d_a,d_b,d_c)$ of the polynomial basis vectors described in \cref{tab:HH bases}. 
    The lower bounds are $2\pi/\sqrt{B}$, where $B$ is the smallest feasible value we could validate (to the precision shown) for the SDP~\cref{eq:sos optimisation problem exploiting sym}. The bound in the last row is sharp to all 5 digits, being saturated by the period $T\approx6.02248$ of the $\Pi_7$ and $\Pi_8$ orbits at $H=1/6$. Runtimes on a single core are shown for SDP solutions in quadruple precision arithmetic (including basis pruning and re-solving), and for rational validation.}
    \label{tab:HH}
\end{table}

\subsubsection{Symmetric POs}
\label{sec:HHsym}

We now turn to bounding periods in the set $\mathcal{P}_\Lambda$, which contains periods of all POs that are invariant under \cref{eq:HH D} but do not lie in the invariant subspace of $\Lambda$, which is where $x_1=x_3=0$. These POs exclude $\Pi_7$ and $\Pi_8$, which are the shortest orbits found in the general case but are not invariant under $\Lambda$. Among the eight families of simple orbits discussed by \citet{churchill1979survey}, only $\Pi_1$ and $\Pi_4$ are invariant under $\Lambda$, and $\Pi_4$ is the only one of these that does not lie in the invariant subspace. Thus the set $\mathcal{P}_\Lambda$  contains all periods that $\Pi_4$ assumes for energies $0<H\le 1/6$. The period of $\Pi_4$ increases as $H$ increases, in contrast to $\Pi_7$ and $\Pi_8$, but it too has $T\to 2\pi$ as $H\to 0$. Therefore, the infimum of $\mathcal{P}_\Lambda$ can be no larger than $2\pi$. We show below that it is exactly $2\pi$.

Our approach broadly follows that of \cref{sec:HHgeneral}, but now we apply \cref{thm:symmetric POs} instead of \cref{thm:sym reduce} and use a modification of \cref{alg:general} to find a bound via \cref{rmk: sym SDP}. Rather than explicitly constructing a basis $a$ matching the requirements of \cref{thm:symmetric POs}, we will start with a basis that vanishes on the symmetric subspace $\Omega_\Lambda$ but might also vanish elsewhere. Following an approximate SDP solution, this basis is then pruned via the same iterative process as in \cref{sec:HHgeneral}. After successful validation, we verify that the final basis $a^o$ satisfies the requirement that $a^o(x)=0$ for $x\in\Omega$ if and only if $x\in\Omega_\Lambda$.

The other bases $(b_0^e,b_0^o,b_1^e,b_1^o)$ are much simpler than in \cref{sec:HHgeneral}, since we no longer have the complication that $a$ automatically vanishes at the nonzero fixed points. We therefore let $b_0^e$ and $b_0^o$ contain nonconstant monomials up to degree $d_b$, let $b_1^e$ and $b_1^o$ contain nonconstant monomials up to degree $d_b-2$, and let $c$ contain nonconstant monomials up to degree $d_c$. As in \cref{sec:HHgeneral}, these bases are starting points for iterative pruning that removes zero eigenvalues from the matrices. These bases are summarized in \cref{tab:HH bases symm}.

\begin{table}[tp]
    \centering
    \renewcommand{\arraystretch}{1.3}
\begin{tabular}{
  >{\centering\arraybackslash}m{1.2cm}
  >{\centering\arraybackslash}m{2cm}
  >{\centering\arraybackslash}m{3cm}
  >{\centering\arraybackslash}m{6.5cm}
}
  \hline Basis & Maximum degree & Purpose & Terms \\\hline
         $a^o$ & $d_a$ & Defining $\Phi(x)$ & Odd monomials of degrees 1 to $d_a$\\
         $b_0^e$ & $d_b$ & Enforcing nonnegativity & Even monomials of degrees 1 to $d_b$ \\
         $b_0^o$ & $d_b$ & Enforcing nonnegativity & Odd monomials of degrees 1 to $d_b$ \\
         $b_1^e$ & $d_b-2$ & Enforcing $H\le 1/6$ & Even monomials of degrees 1 to $d_b-2$ \\
         $b_1^o$ & $d_b-2$ & Enforcing $H\le 1/6$ & Odd monomials of degrees 1 to $d_b-2$ \\
         $c$ & $d_c=2d_b-1$ & Defining $V(x)$ & Even monomials of degrees 1 to $d_c$\\\hline
    \end{tabular}
    \caption{Polynomial basis vectors used in the constraint \cref{eq:sdp equality exploiting sym} to bound periods of symmetric POs in the H\'enon--Heiles system \cref{eq:hh f}. These vectors are starting points for an iterative procedure (see text) that prunes terms with near-zero coefficients and re-solves the SDP~\cref{eq:sdp equality exploiting sym}.}
    \label{tab:HH bases symm}
\end{table}

\begin{table}[tp]
    \centering
\begin{tabular}{
  >{\centering\arraybackslash}m{.7cm}
  >{\centering\arraybackslash}m{.7cm}
  >{\centering\arraybackslash}m{.7cm}
  >{\centering\arraybackslash}m{1cm}
  >{\centering\arraybackslash}m{2.8cm}
  >{\centering\arraybackslash}m{1.4cm}
  >{\centering\arraybackslash}m{2cm}
}        \hline$d_a$&$d_b$&$d_c$&$B$&\textbf{Lower bound on period}&SDP time (s)&Validation time (s)\\\hline
        1&3&5&3.00& \textbf{3.6275}& 0.55 &0.01\\
        2&3&5&3.00& \textbf{3.6275}& 0.60 &0.01\\
        1&4&7&3.00& \textbf{3.6275}& 8.61 &0.25\\
        2&4&7&1.07& \textbf{6.0741}& 11.5 &0.26\\
        3&4&7&1.00& \textbf{$\mathbf{2}\boldsymbol{\pi}$}&9.67 &0.29 \\\hline
    \end{tabular}
    \caption{Validated lower bounds on periods of symmetric POs of the H\'enon--Heiles system, for different maximum degrees $(d_a,d_b,d_c)$ of the polynomial basis vectors described in \cref{tab:HH bases symm}. Bounds apply to POs that are invariant under the sign change~\cref{eq:HH D} as curves but not at every point. The bound of $2\pi$ is sharp, being approached by the $\Pi_4$ orbit as $H\to 0$. The lower bounds are $2\pi/\sqrt{B}$, where $B$ is the smallest value we could validate (to the precision shown) for the SDP~\cref{eq:sos optimisation problem exploiting sym}. Runtimes on a single core are shown for SDP solutions in quadruple precision arithmetic (including basis pruning and re-solving), and for rational validation.}
    \label{tab:HH symm}
\end{table}

\Cref{tab:HH symm} reports rigorously validated lower bounds on periods of POs that are symmetric -- but are not confined to the invariant subspace $\Omega_\Lambda$ -- which we have computed using several different choices for maximum polynomial degrees. Our best lower bound of $T\ge 2\pi$, which is obtained by validating feasibility of the SDP when $B=1$, is an exactly sharp bound on the infimum of $\mathcal{P}_\Lambda$ since the $\Pi_4$ orbit approaches this period as $H\to0$. This bound is computed and validated in about 10 seconds on a single core. For each $(d_a,d_b)$, we choose $d_c$ so that the polynomial degree of $\ld V$ is equal to the maximum degree of the other terms in \cref{eq:sdp equality exploiting sym}.

As discussed above, one final step is needed to confirm that this bound applies to all symmetric POs not in $\Omega_\Lambda$. We must check that the pruned basis $a^o$, which vanishes on $\Omega_\Lambda$ by construction, does not vanish elsewhere in $\Omega$. We have checked this property for all cases in \cref{tab:HH symm}, but for brevity we demonstrate it only for the $d_a=4$ case that gives the sharp bound. After the iterative reduction process, the algorithm successfully validated the SDP using the pruned basis
\begin{multline}
a^o(x) = \big(
    x_3x_4,\,
    -x_3+x_2x_3,\,
    x_1x_4,\,
    -x_1+x_1x_2,\,
    x_2x_3x_4,\,
    -x_3+x_2^2x_3,\\
    x_1x_4^2,\,
    x_1x_3^2,\,
    x_1x_2x_4,\,
    -x_1+x_1x_2^2,\,
    x_1^2x_3,\,
    x_1^3,\,
    -x_1+x_1x_2^3,\,
    x_1^3x_2\big)
\end{multline}
It is clear that $a^o$ vanishes on the invariant subspace $\Omega_\Lambda$ where $x_1=x_3=0$. For the converse, we suppose that $a^o(x)=0$ and show $x_1=x_3=0$. Since $x_1^3$ is an entry of $a^o$, we have $x_1=0$. Since $x_3x_4$ is an entry, $x_3=0$ or $x_4=0$. In the second case, where $x_1=x_4=0$, we have $x_3=0$ or $x_2=1$ since $-x_3+x_2x_3$ and $-x_3+x_2^2x_3$ are entries of $a^o$. However, $x_3$ must vanish even if $x_2=1$ since, when $x_1=x_4=0$ and $x_2=1$, the energy constraint that defines $\Omega$ becomes 
\begin{equation}
H(x) = \frac{1}{6} + \frac{1}{2}x_3^2 \le \frac{1}{6}.
\end{equation}
Therefore, $a^o$ vanishing implies $x_1=x_3=0$. Similar arguments can be made for all $a^o$ bases used to produce the bounds reported in \cref{tab:HH symm}, so these bounds indeed apply to all symmetric POs not confined to $\Omega_\Lambda$.

\subsection{Minimal periods in the Lorenz system}
\label{sec:lorenz}

The \citet{lorenz1963deterministic} system is governed by an ODE system $\ddt{}x(t)=f(x(t))$ with the polynomial right-hand side
\begin{equation}
\label{eq:lorenz}
f(x) = \big(\sigma(x_2-x_1),x_1(\rho-x_3)-x_2,x_1 x_2 - \beta x_3\big). 
\end{equation}
We fix the standard parameter values of $(\beta,\sigma,\rho)=(8/3,10,28)$, at which an infinite number of unstable POs are embedded in a chaotic attractor. Validated numerics have been used to prove the existence of many POs that wind around the two wings of the attractor in particular sequences, and the periods of these POs have been validated to high precision \cite{barrio2015database}. The shortest PO among these has a period $T\approx1.55865$, whose precise value has been validated to 1000 digits. \Cref{fig:lorenz} shows this PO, which is often called the LR orbit because it winds once each around the `left' and `right' wings of the attractor \citep{sparrow2012lorenz}. This orbit is invariant under the sign symmetry of the Lorenz system, which is 
\begin{equation}
\Lambda=\begin{bmatrix}-1&0&0\\0&-1&0\\0&0&1\end{bmatrix}.
\label{eq:lorenzD}
\end{equation}
The LR orbit is the shortest PO embedded in the chaotic attractor \cite{barrio2015database}. As far as we know, no authors have suggested that POs exist outside of this attractor at the standard parameters, but nor have they ruled out such POs. The bounds we report below confirm that, even if POs exist outside of the attractor, their periods cannot be shorter than the approximate period $T\approx1.55865$ of the LR orbit.

\begin{figure}
    \centering
    \begin{tikzpicture}
\begin{groupplot}[
    group style={group size=2 by 1, horizontal sep=2.0cm, vertical sep=1.8cm},
    width=7.2cm,
    height=5.6cm,
    axis equal image,
    grid=both,
    grid style={line width=.1pt, draw=gray!25},
    major grid style={line width=.2pt, draw=gray!45},
    tick label style={font=\small},
    label style={font=\small},
    title style={font=\small},
    ylabel style={
        rotate=-90,
        xshift=0.4cm
    },
    xmin=-20, xmax=20,
    ymin=-20, ymax=20,
    xtick={-15,0,15},
    ytick={-15,0,15,30,45},
]

\nextgroupplot[
    xlabel={$x_1$},
    ylabel={$x_2$},
]
\addplot+[no marks, thick] table[x index=1, y index=2] {lorenz_LR.dat};

\nextgroupplot[
    xlabel={$x_1$},
    ylabel={$x_3$},
    ymin=5, ymax=45,
]
\addplot+[no marks, thick] table[x index=1, y index=3] {lorenz_LR.dat};

\end{groupplot}
\end{tikzpicture}
\caption{The LR orbit of the Lorenz equations at the standard parameter values for chaos, projected onto two different coordinate planes. Its period of $T\approx1.55865$ is the shortest among all known POs.}
    \label{fig:lorenz}
\end{figure}

\subsubsection{Bounds on periods of general POs}
\label{sec:lorenzgeneral}

This subsection presents lower bounds on the periods of all POs in the Lorenz system, computed by applying \cref{thm:sym reduce}. Before formulating the computations, we rescale $x$ by $25$ to improve numerical conditioning of the resulting SDPs as in~\cite{goluskin2018bounding}. We also rescale $t$ by 3 to obtain an ODE with integer coefficients as in~\cite{parker2024lorenz}. The rescaling of time makes all periods 3 times smaller, so lower bounds for the rescaled system must be multiplied by 3 to apply to the standard formulation~\cref{eq:lorenz}. The rescaled Lorenz equations have right-hand side
\begin{equation}
\label{eq:lorenzrescaled}
f(x)=\big(30(x_2-x_1), 84x_1-75x_1x_3-3x_2, 75x_1 x_2 - 8 x_3\big).
\end{equation}
Although there are many compact sets known to contain all POs~\cite{doering1995shape,goluskin2020bounding}, it suffices to formulate our bounding computations for the domain $\Omega=\mathbb{R}^3$, in which case $I=J=0$ in \cref{eq:sdp equality exploiting sym}.

\Cref{tab:lorenzgeneralbases} summarizes our choices for the polynomial basis vectors appearing in the relation~\cref{eq:sdp equality exploiting sym} that gives the equality constraint of the SDP~\cref{eq:sdp equality exploiting sym}. As explained in \cref{sec:HHgeneral}, we want all basis vectors to vanish at any fixed points, so that the matrices in~\cref{eq:sdp equality exploiting sym} can be strictly positive definite. And, as in \cref{sec:HHgeneral}, this is automatic for the vectors that are Lie derivatives, but we must ensure it for $b^e$ and $b^o$. There is one fixed point at the origin, where both vectors vanish if $b^e$ contains no constant terms. With the Lorenz system rescaled as~\cref{eq:lorenzrescaled}, the other two fixed points are at
\begin{equation}
(x_1,x_2,x_3)=\left(\pm\sqrt{72},\pm\sqrt{72},27\right)/25.
\end{equation}
We ensure vanishing at these points as in \cite{parker2024lorenz}: entries in $b^e$ and $b^o$ consist of even and odd monomials, respectively, multiplied by $25^2x_1^2-72$, $25^2x_2^2-72$ or $25x_3-27$. Monomials are included that yield entries of degree no larger than $d_b$.

\begin{table}[tp]
    \centering
    \renewcommand{\arraystretch}{1.6}
\begin{tabular}{
  >{\centering\arraybackslash}m{1.2cm}
  >{\centering\arraybackslash}m{1.6cm}
  >{\centering\arraybackslash}m{3cm}
  >{\centering\arraybackslash}m{7cm}
}
  \hline Basis & Maximum degree & Purpose & Terms \\\hline
         $a^e$ & $d_a$ & Defining $\Phi(x)$ & $\ld w^e$, where $w^e$ contains even monomials of degrees 1 to $d_a-2$, and $30d_ax_1^{d_a-2}x_3-75x_1^{d_a}$\\
         $a^o$ & $d_a$ & Defining $\Phi(x)$ & $\ld w^o$, where $w^o$ contains odd monomials of degrees 1 to $d_a-2$, and $x_1^{d_a-2k-1}\left(x_2^2+x_3^2\right)^k$ for $k=0,\dots,d_a/2$\\
         $b_0^e$ & $d_b=d_a$ & Enforcing nonnegativity & Even monomials of degrees 1 to $d_b-2$, multiplied by $25^2x_1^2-72$; \newline even monomials of degrees 1 to $d_b-2$, multiplied by $25^2x_2^2-72$; \newline odd monomials of degrees 1 to $d_b-1$, multiplied by $25x_3-27$\\
         $b_0^o$ & $d_b=d_a$ & Enforcing nonnegativity & Odd monomials of degrees 1 to $d_b-2$, multiplied by $25^2x_1^2-72$; \newline odd monomials of degrees 1 to $d_b-2$, multiplied by $25^2x_2^2-72$; \newline even monomials of degrees 1 to $d_b-1$, multiplied by $25x_3-27$\\
         $c$ & $d_c=2d_a$ & Defining $V(x)$ & Even monomials of degrees 1 to $d_c-1$, \newline and $x_1^{d_c-2k}(x_2^2+x_3^2)^k$ for $k=0,\dots,d_c/2$\\\hline
    \end{tabular}
    \caption{Polynomial basis vectors used in the constraint \cref{eq:sdp equality exploiting sym} to bound periods of general POs in the rescaled Lorenz system \cref{eq:lorenzrescaled}.}
    \label{tab:lorenzgeneralbases}
\end{table}

Finally, the entries of $w^e$ and $w^o$ are chosen to obey certain algebraic constraints, which are not needed in the the H\'enon--Heiles example of \cref{sec:HH}. In particular, we would like $a^e=\ld w^e$ to have the same maximum degree $d_a$ as its Lie derivative $\ld a^e$, and likewise for the odd vector $a^o$. This is so that we do not have redundant, nonpositive, high-degree terms on the left-hand side of \cref{eq:sdp equality exploiting sym}. Note that an arbitrary polynomial $p$ of maximum degree $d$ can have a Lie derivative $\ld p$ of maximum degree of $d+1$, since $f$ of the Lorenz system is quadratic. The only way for $p$ and $\ld p$ to have the same maximum degree is if the degree-$(d+1)$ contributions to $\ld p$ cancel each other. As explained and used previously \cite{swinnerton2001bounds,goluskin2018bounding}, such cancellation dictates that the highest-degree terms of $p$ are products of powers of $x_1$ and of $(x_2^2+x_3^2)$. Here we must go one step further because we do not specify $p$ directly; we must specify $w^e$ and $w^o$ such that their first and second derivatives have the same maximum degree $d_a$. It suffices to consider $d_a\geq3$, so that $w^e$ and $w^o$ contain the even and odd $x_n$ components, respectively, as needed to apply \cref{thm:sym reduce}. Components of $w^e$ and $w^o$ with degree $d_a-1$ are restricted to the form
\begin{equation}
    w_i=x_1^{d_a-2k-1}\left(x_2^2+x_3^2\right)^k
    \label{eq:first deriv cancellation}
\end{equation}
for some $k$. These components lead to degree-$(d_a-1)$ terms in $a_i=\ld w_i$, and degree-$d_a$ terms in $\ld a_i$. There is only one degree-$d_a$ component, which belongs to $w^e$ or $w^o$ if $d_a$ is even or odd, respectively. This component is
\begin{equation}
    w_i = 30d_ax_1^{d_a-2}x_3-75x_1^{d_a},
\end{equation}
which gives $x_1^{d_a}$ as the only degree-$d_a$ term in $a_i=\ld w_i$, so that $\ld a_i$ also has degree $d_a$. The basis $c$ used to construct the auxiliary function $V$ has highest-degree terms with a similar special form to~\cref{eq:first deriv cancellation}, exactly following \cite{goluskin2018bounding} to give cancellations in the first derivative only.

\Cref{tab:lorenzgeneral} reports rigorously validated lower bounds on orbit periods that we computed using a version of \cref{alg:general} with the polynomial basis vectors summarized in \cref{tab:lorenzgeneralbases}. In contrast to the examples of \cref{sec:HH}, we do not need to iteratively prune the basis vectors in~\cref{tab:lorenzgeneralbases}. As in the other examples, our best lower bound is sharp to five digits. The computation giving this best bound required around 11 minutes to solve the SDP, followed by roughly one hour for validation. Relative to the bounds on general and symmetric orbits in~\cref{sec:HH}, bounds of similar precision are harder to compute for general orbits of the Lorenz system. This is because the basis vectors are of higher polynomial degree, perhaps because the shortest PO has a more complicated shape. In all cases we chose $d_a=d_b$ and $d_c$ twice as large, so that all terms in \cref{eq:sdp equality exploiting sym} have the same maximum degree.

\begin{table}[tp]
    \centering
\begin{tabular}{
  >{\centering\arraybackslash}m{.7cm}
  >{\centering\arraybackslash}m{.7cm}
  >{\centering\arraybackslash}m{.7cm}
  >{\centering\arraybackslash}m{1cm}
  >{\centering\arraybackslash}m{2.8cm}
  >{\centering\arraybackslash}m{1.4cm}
  >{\centering\arraybackslash}m{2cm}
}        \hline$d_a$&$d_b$&$d_c$&$B$&\textbf{Lower bound on period}&SDP time (s)&Validation time (s)\\\hline
        3 & 3 & 6 & 1128 & \textbf{0.5612}&0.06&0.03\\
        4 & 4 & 8 & 888 & \textbf{0.6325}&0.44&0.62\\
        5 & 5 & 10& 488 & \textbf{0.8532}&4.01&8.98\\
        6 & 6 & 12& 325 & \textbf{\underline{1.}0455}&29.2&79.7\\
        7 & 7 & 14 & 155.9 & \textbf{\underline{1.5}096}&81.5&573\\
        8 & 8 & 16 & 146.26 & \textbf{\underline{1.5586}}&641&3558\\\hline
    \end{tabular}
    \caption{Validated lower bounds on orbit periods for the Lorenz system~\cref{eq:lorenz}, for different maximum degrees $(d_a,d_b,d_c)$ of the polynomial basis vectors described in \cref{tab:lorenzgeneralbases}. Underlined digits are sharp, being saturated by the period $T\approx1.55865$ of the LR orbit shown in \cref{fig:lorenz}. The lower bounds are $6\pi/\sqrt{B}$, which accounts for the fact that $B$ is validated for the SDP~\cref{eq:sos optimisation problem exploiting sym} with the rescaled Lorenz system~\cref{eq:lorenzrescaled}, where time is 3 times faster. Each $B$ is the smallest we could validate to the precision shown. Runtimes on a single core are shown for SDP solutions in quadruple precision arithmetic, and for rational validation.}
    \label{tab:lorenzgeneral}
\end{table}

\subsubsection{Symmetric POs}
\label{sec:lorenzsym}

Since the LR orbit is symmetric, its period is included in the set $\mathcal{P}_\Lambda$, which contains periods of all POs that are invariant under \cref{eq:lorenzD}. (None of these invariant POs are confined to the invariant subspace $\Omega_\Lambda$ where $x_1=x_2$.) Therefore, $\mathcal P_\Lambda$ and the set $\mathcal P$ of all periods should have the same infimum. In contrast to the examples of \cref{sec:HH}, our formulation \cref{thm:symmetric POs} for bounding periods of only $\Lambda$-invariant POs cannot give a bound larger than the period (of the LR orbit) that saturates bounds from the general formulation \cref{thm:general}. Nonetheless we have computed bounds using the formulation \cref{thm:symmetric POs} also. This provides a further test of our formulation for symmetric POs, and it illustrates faster convergence with polynomial degree, albeit with a formulation whose bounds are not guaranteed to apply to non-symmetric orbits.

\Cref{tab:lorenzsymmbases} summarizes the bases we use to compute verified lower bounds on $\mathcal P_\Lambda$. Our approach is as in \cref{sec:lorenzgeneral}, but with \cref{thm:symmetric POs} used in place of \cref{thm:sym reduce}, and with \cref{alg:general} modified to use \cref{rmk: sym SDP}. The $a^o$ vector includes $x_1$ and $x_2$ as components, so that it vanishes only on the invariant subspace $\Omega_\Lambda$, as required by \cref{thm:symmetric POs}.

\begin{table}[tp]
    \centering
    \renewcommand{\arraystretch}{1.3}
\begin{tabular}{
  >{\centering\arraybackslash}m{1.2cm}
  >{\centering\arraybackslash}m{1.5cm}
  >{\centering\arraybackslash}m{3cm}
  >{\centering\arraybackslash}m{7cm}
}
  \hline Basis & Maximum degree & Purpose & Terms \\\hline
         $a^o$ & $d_a$ & Defining $\Phi(x)$ & Odd monomials of degrees 1 to $d_a$\\
         $b^e$ & $d_b$ & Enforcing nonnegativity & Even monomials of degrees 1 to $d_b$\\
         $b^o$ & $d_b$ & Enforcing nonnegativity & Odd monomials of degrees 1 to $d_b$ \\
         $c$ & $d_c=2d_b$ & Defining $V(x)$ & Even monomials of degrees 1 to $d_c-1$, and $x_1^{d_c-2k}(x_2^2+x_3^2)^k$ for $k=0,\dots,d_c/2$\\\hline
    \end{tabular}
    \caption{Polynomial basis vectors used in the constraint \cref{eq:sdp equality exploiting sym} to bound periods of symmetric POs in the rescaled Lorenz system \cref{eq:lorenzrescaled}.}
    \label{tab:lorenzsymmbases}
\end{table}

\Cref{tab:lorenz symm} reports rigorously validated lower bounds on periods of symmetric POs that we computed using a version of \cref{alg:general} with the polynomial basis vectors summarized in \cref{tab:lorenzsymmbases}. As with the bounds for general POs reported in \cref{sec:lorenzgeneral}, the best bound in \cref{tab:lorenz symm} is saturated to 5 digits by the LR orbit. In the present case, however, this is achieved with lower polynomial degrees and much lower computational cost than for general POs (cf.\ \cref{tab:lorenzgeneral}), at the expense of yielding no statement about non-symmetric POs.

\begin{table}[tp]
    \centering
\begin{tabular}{
  >{\centering\arraybackslash}m{.7cm}
  >{\centering\arraybackslash}m{.7cm}
  >{\centering\arraybackslash}m{.7cm}
  >{\centering\arraybackslash}m{1cm}
  >{\centering\arraybackslash}m{2.8cm}
  >{\centering\arraybackslash}m{1.4cm}
  >{\centering\arraybackslash}m{2cm}
}        \hline$d_a$&$d_b$&$d_c$&$B$&\textbf{Lower bound on period}&SDP time (s)&Validation time (s)\\\hline
        2 & 3 &6& 705 & \textbf{0.7099}&0.05&0.03\\
        2 & 4&8 & 310 & \textbf{\underline{1.}0705}&0.31&0.21\\
        3 & 4&8 & 185 & \textbf{\underline{1.}3858}&0.34&0.34\\
        3 &  5&10 & 146.52 & \textbf{\underline{1.55}72}&2.54&2.38\\
        4 & 5&10 & 146.33 & \textbf{\underline{1.558}2}&2.29&2.99\\
        4 &  6&12 & 146.28 & \textbf{\underline{1.558}5}&11.6&14.0\\
        5 &  6 &12& 146.26 & \textbf{\underline{1.5586}}&8.51&20.5\\\hline
    \end{tabular}
    \caption{Validated lower bounds on periods of symmetric POs of the Lorenz system, for different maximum degrees $(d_a,d_b,d_c)$ of the polynomial basis vectors described in \cref{tab:lorenzsymmbases}. Underlined digits are sharp, being saturated by the period $T\approx1.55865$ of the LR orbit. The lower bounds are $6\pi/\sqrt{B}$, which accounts for the fact that $B$ is validated for the SDP~\cref{eq:sos optimisation problem exploiting sym} with the rescaled Lorenz system~\cref{eq:lorenzrescaled}, where time is 3 times faster. Each $B$ is the smallest we could validate to the precision shown. Runtimes on a single core are shown for SDP solutions in quadruple precision arithmetic, and for rational validation.}
    \label{tab:lorenz symm}
\end{table}

\section{Discussion and conclusion}
\label{sec:conc}

We have developed a framework for proving lower bounds on the periods of
POs in autonomous ODEs. This framework combines the idea behind the work of Yorke \cite{yorke1969periods} and its extensions, which are based on Wirtinger's inequality, with methods that use auxiliary functions to bound time averages. The latter addition allows for sharper bounds by making more use of the fact that POs integrate the vector field. This fact typically prevents any PO from, say, attaining the worst-case Lipschitz constant at all times, which is the estimate on which Yorke's bound~\cite{yorke1969periods} is based. For any particular ODE, maximizing lower bounds on PO periods within our framework leads naturally to optimization problems.
For polynomial ODEs, these optimization problems can be relaxed into convex polynomial optimization problems with sum-of-squares constraints, whose optimal values can be computed by numerically solving semidefinite programs. For ODEs with symmetries, we have described how symmetry can be used to save computational cost when solving SDPs, as well as how our framework can be modified so that lower bounds need not apply to non-symmetric POs. Finally, we have described an algorithm for validating the numerical SDP solutions using rational arithmetic, yielding rigorous SOS certificates of numerical lower bounds on periods of POs.

The computations reported in \cref{sec:examples} demonstrate that validated numerics using our methods can give very sharp bounds for POs in examples that display either Hamiltonian or dissipative chaos. In the Hamiltonian H\'enon--Heiles system with \(H\leq 1/6\), our best lower bound for all POs is saturated to five digits by the so-called $\Pi_7$ and $\Pi_8$ orbits at \(H=1/6\). When the formulation is modified to apply only to POs with a sign symmetry, the bounds become larger and instead prove the sharp bound \(T\geq 2\pi\). For the dissipative Lorenz system at the standard parameters, our best validated lower bounds for both general and symmetric POs are saturated to five digits by the well-known PO that transits each wing of the Lorenz attractor once before closing. This PO was known to be the shortest one embedded in the strange attractor, but our bounds also rule out shorter POs outside of this attractor.

Our bounding framework appears to be sharp in all of our examples, in the sense that bounds come arbitrarily close to the true minimal period as SOS computations are repeated using polynomial bases of increasing dimension. However, we have not been able to prove this sharpness for, say, all ODEs with differentiable right-hand sides. The relaxations in the polynomial case to finite-dimensional bases and to SOS conditions are well understood in similar contexts, and their convergence with increasing polynomial degree can be shown by standard arguments under various assumptions. The seemingly harder question is whether our general bounding formulation is sharp at the level of infinite-dimensional auxiliary functions. In particular, one may hope to prove a converse of \cref{thm:general} for a broad class of ODE right-hand sides $f$. The following conjecture is one possible form of such a converse.

\begin{conjecture}
Let $\Omega\subset\mathbb R^n$ be a compact forward-invariant set for a $C^2$ vector field $f$, and let $\mathcal P$ be the nonempty set of all PO periods. Assume that every invariant probability measure not supported on fixed points lies in the weak-* closure of the convex hull of the periodic-orbit measures. Then, for every $B>(2\pi/T^*)^2$ where $T^*>0$ is a lower bound on $\mathcal P$, there exist $m\in\mathbb N$, $\mathsf Q\succ0$, $V\in C^1(\Omega)$, and $w\in C^2(\Omega,\mathbb R^m)$ such that
\begin{equation}
\label{eq:conj}
    B\,a(x)^{\mathsf T}\mathsf Q a(x)
    -
    (\ld a(x))^{\mathsf T}\mathsf Q(\ld a(x))
    +
    \ld V(x)
    \geq 0
    \qquad \forall x\in\Omega,
\end{equation}
where $a=\ld w$, and where $n$ of the components of $w$ are the components of $x$.
\end{conjecture}
The assumption about invariant measures in the conjecture may be needed for the following reason. If \cref{eq:conj} holds, then $B$ is sufficiently large that $\int_\Omega \left[Ba^{\mathsf T}\mathsf Q a-(\ld a)^{\mathsf T}\mathsf Q(\ld a)\right] \,\mathrm d\mu \ge 0$ for every $\mu$ in the convex space of invariant probability measures on $\Omega$, which is compact in the weak-* topology. This includes invariant measures supported on structures other than periodic orbits, such as invariant tori. Such measures could obstruct sharp $B$ values in general, but not if they can be approximated by periodic orbit measures. The assumption is true, for instance, if $\Omega$ is a basic set of an Axiom A flow \citep{sigmund1972space}. However, even if the conjecture can be proved, given a particular $f$ it may be challenging to determine whether the assumption about invariant measures is satisfied.

If an optimal or near-optimal bound $B$ can be certified by finding $(w,\mathsf Q,V)$ that satisfy~\cref{eq:conj}, this certificate also contains information about POs with minimal or near-minimal period. For simplicity, suppose there exists a PO attaining a minimal period of $T^*$, and that \cref{eq:conj} holds with the exactly sharp value $B^*=(2\pi/T^*)^2$. The left-hand side of~\cref{eq:conj} is nonnegative for all $x\in\Omega$, and its time integral over any PO of period $T^*$ is zero, so this expression must vanish pointwise on any such PO. Furthermore, if a bound $B$ and PO are not exactly optimal but are near-optimal, the expression~\cref{eq:conj} must be small on that PO in a quantitative sense~\cite{tobasco2018optimal}. This observation gives a way to search systematically for POs with minimal period. After SOS computations give nearly sharp bounds on PO periods, any $x\in\Omega$ where the left-hand side of~\cref{eq:conj} is relatively small can be a starting point for an iterative method that finds POs, and which may converge to the shortest PO. This approach has been used to find POs that maximize or minimize time averages of various quantities~\cite{tobasco2018optimal,lakshmi2020finding}, although not to find POs of minimal period.

There are several promising directions for generalizing the present work. First, our formulation for bounds on periods of symmetric POs was given only for sign symmetries, but it may well be generalized to other finite symmetry groups by decomposing the
polynomial spaces into irreducible representations \cite{gatermann2004symmetry}. This would enable, for example, bounds applying only to POs possessing the threefold rotational symmetry of the H\'enon--Heiles system studied in \cref{sec:HH}. Second, polynomial optimization methods can be extended to polynomials of other algebraically closed classes of functions, such as rational, trigonometric, or exponential functions, or by using polynomial or rational approximations of non-polynomial vector fields with rigorous error bounds.  Third, the methods presented here for ODEs may be augmented to provide lower bounds on periods of POs in dissipative PDEs. An approach that has succeeded for related questions about PDEs~\cite{goluskin2019bounds,fuentes2022global} is to consider an ODE obtained by projecting the PDE onto a finite-dimensional space, rigorously estimate the differences between this ODE and the original PDE, and incorporating these estimates into SOS optimization formulations like those we have used here. By such a computer-assisted procedure, results like the positive lower bound on periods in the Navier--Stokes equations~\cite{kukavica1994absence} could be made quantitative in a precise way.

\section*{Data availability statement}
Code to produce all the results in this paper is available at \\\url{https://github.com/jeremypparker/ODE_Period_Bounds}.

\section*{Acknowledgements}
The authors wish to thank James Robinson, Warwick Tucker and Ian Tobasco for helpful discussions and references. During this work DG was supported by the NSERC Discovery Grants Program (awards RGPIN-2018-04263 and RGPIN-2025-06823). JP's travel for this collaboration was supported by the QJMAM Fund for Applied Mathematics. ChatGPT 5.5 was used to generate code for converging the periodic orbits discussed in \cref{sec:HH,sec:lorenz}, which are already known in the literature, and for producing \cref{fig:hh,fig:lorenz}.

\bibliographystyle{abbrvnat_nourl}
\bibliography{references.bib}

\end{document}